\documentclass[notitlepage,11pt,reqno]{amsart}
\usepackage[foot]{amsaddr}
\usepackage[numbers,sort]{natbib}

\usepackage{amssymb,nicefrac,bm,upgreek,mathtools,verbatim}
\usepackage[final]{hyperref}
\usepackage[mathscr]{eucal}
\usepackage{dsfont}

\usepackage{color}
\definecolor{dmagenta}{rgb}{.4,.1,.5}
\definecolor{dblue}{rgb}{.0,.0,.5}
\definecolor{mblue}{rgb}{.0,.0,.8}
\definecolor{ddblue}{rgb}{.0,.0,.4}
\definecolor{dred}{rgb}{.6,.0,.0}
\definecolor{dgreen}{rgb}{.0,.5,.0}
\definecolor{Eeom}{rgb}{.0,.0,.5}

\usepackage[normalem]{ulem}

\usepackage[margin=1in]{geometry}
\allowdisplaybreaks

\newtheorem{lemma}{Lemma}[section]
\newtheorem{theorem}{Theorem}[section]

\newtheorem{corollary}{Corollary}[section]

\theoremstyle{definition}
\newtheorem{definition}{Definition}[section]

\newtheorem{example}{Example}[section]
\newtheorem{remark}{Remark}[section]

\numberwithin{equation}{section}

\hypersetup{
  colorlinks=true,
  citecolor=dblue,
  linkcolor=dblue,
  frenchlinks=false,
  pdfborder={0 0 0},
  naturalnames=false,
  hypertexnames=false,
  breaklinks}
\usepackage[capitalize,nameinlink]{cleveref}

\crefname{section}{Section}{Sections}
\crefname{subsection}{Subsection}{Subsections}
\crefname{condition}{Condition}{Conditions}
\crefname{hypothesis}{Hypothesis}{Conditions}
\crefname{assumption}{Assumption}{Assumptions}
\crefname{lemma}{Lemma}{Lemmas}
\crefname{claim}{Claim}{Claims}

\Crefname{figure}{Figure}{Figures}

\crefformat{equation}{\textup{#2(#1)#3}}
\crefrangeformat{equation}{\textup{#3(#1)#4--#5(#2)#6}}
\crefmultiformat{equation}{\textup{#2(#1)#3}}{ and \textup{#2(#1)#3}}
{, \textup{#2(#1)#3}}{, and \textup{#2(#1)#3}}
\crefrangemultiformat{equation}{\textup{#3(#1)#4--#5(#2)#6}}%
{ and \textup{#3(#1)#4--#5(#2)#6}}{, \textup{#3(#1)#4--#5(#2)#6}}%
{, and \textup{#3(#1)#4--#5(#2)#6}}

\Crefformat{equation}{#2Equation~\textup{(#1)}#3}
\Crefrangeformat{equation}{Equations~\textup{#3(#1)#4--#5(#2)#6}}
\Crefmultiformat{equation}{Equations~\textup{#2(#1)#3}}{ and \textup{#2(#1)#3}}
{, \textup{#2(#1)#3}}{, and \textup{#2(#1)#3}}
\Crefrangemultiformat{equation}{Equations~\textup{#3(#1)#4--#5(#2)#6}}%
{ and \textup{#3(#1)#4--#5(#2)#6}}{, \textup{#3(#1)#4--#5(#2)#6}}%
{, and \textup{#3(#1)#4--#5(#2)#6}}

\crefdefaultlabelformat{#2\textup{#1}#3}

\makeatletter
\DeclareRobustCommand\widecheck[1]{{\mathpalette\@widecheck{#1}}}
\def\@widecheck#1#2{%
    \setbox\z@\hbox{\m@th$#1#2$}%
    \setbox\tw@\hbox{\m@th$#1%
       \widehat{%
          \vrule\@width\z@\@height\ht\z@
          \vrule\@height\z@\@width\wd\z@}$}%
    \dp\tw@-\ht\z@
    \@tempdima\ht\z@ \advance\@tempdima2\ht\tw@ \divide\@tempdima\thr@@
    \setbox\tw@\hbox{%
       \raise\@tempdima\hbox{\scalebox{1}[-1]{\lower\@tempdima\box
\tw@}}}%
    {\ooalign{\box\tw@ \cr \box\z@}}}

\def\subsection{\@startsection{subsection}{0}%
\z@{\linespacing\@plus\linespacing}{\linespacing}%
{\bf}}

\makeatother

\newcommand{\df}{\coloneqq}
\DeclareMathOperator{\Exp}{\mathbb{E}} 
\DeclareMathOperator{\Prob}{\mathbb{P}} 
\newcommand{\D}{\mathrm{d}}          
\newcommand{\RR}{\mathbb{R}}         
\newcommand{\Rd}{{\mathbb{R}^d}}       
\newcommand{\Ind}{\mathds{1}}            

\newcommand{\grad}{\nabla}

\newcommand{\sB}{\mathscr{B}}    

\newcommand{\cC}{\mathcal{C}}     

\newcommand{\sF}{\mathfrak{F}}    
\newcommand{\cG}{\mathcal{G}}
\newcommand{\cK}{\mathcal{K}}

\newcommand{\sR}{\mathscr{R}}
\newcommand{\cX}{\mathcal{X}}    
\newcommand{\cY}{\mathcal{Y}}

\newcommand{\norm}[1]{\lVert#1\rVert}

\DeclareMathOperator{\diam}{diam}

\DeclareMathOperator{\Psidel}{\Psi(--\Delta)}

\begin{document}

\title[Nonlocal logistic equation with harvesting]%
{\sc \textbf{A study of nonlocal spatially heterogeneous logistic equation with harvesting}}

\author{Anup Biswas and Mitesh Modasiya}

\address{
Department of Mathematics, Indian Institute of Science Education and Research, Dr. Homi Bhabha Road,
Pune 411008, India. Email: anup@iiserpune.ac.in; mitesh.modasiya@students.iiserpune.ac.in }

\date{}


\begin{abstract}
We study a class of nonlocal reaction-diffusion equations with a harvesting term where
the nonlocal operator is given by
 a Bernstein function of the Laplacian. In particular, it
includes the fractional Laplacian, fractional relativistic operators, sum of fractional
Laplacians of different order etc. We study existence, uniqueness and multiplicity results of the solutions to the steady state equation. We also consider the parabolic counterpart and establish the long time asymptotic of the solutions. 
Our proof techniques rely on both analytic and probabilistic arguments.
\end{abstract}
\keywords{Bernstein functions of the Laplacian, nonlocal semipositone problems, long time behaviour, nonlocal Fisher-KPP, bifurcation, variable order nonlocal kernel.}
\subjclass[2010]{Primary: 35R11, 35S15, 35K57  Secondary: 35J60, 92D25}

\maketitle

\section{\bf Introduction}
One of the most celebrated reaction-diffusion models was introduced by Fisher \cite{Fish} and Kolmogorov, Petrovsky, and Piskunov \cite{KPP} in 1937 (popularly known as Fisher-KPP model). Since then, it has been widely used to model spatial propagation or spreading of biological species into homogeneous environments
(see books \cite{M93,OL02} for a review). The corresponding equation is given by
$$(\partial_t-\nu\Delta ) u(x, t)= a u (1-\frac{u}{N})\quad \text{in}\;\; D\times (0, T),
\quad u(x, t)=0\quad \text{on}\; \partial D\times [0, T],$$
where $u=u(x, t)$ represents the population density at the
space-time point $(x,t)$,
$\nu$ is the diffusion parameter, $N>0$ is the carrying capacity of the environment.
Imposing the solution to vanish outside the domain $D$ corresponds to a confinement situation, for instance in a hostile environment. Various generalizations
to the above model have been studied both in bounded and unbounded domains.

However, it is recently observed that the heat operator may be too restrictive to describe the spreading of species and for this reason a nonlocal operator may be more useful than a local one, see for instance Berestycki-Coville-Vo \cite{BCV},
Humphries et al. \cite{H10}, Huston et\ al. \cite{HMMV},
Massaccesi-Valdinoci \cite{MV}, Viswanathan et\ al. \cite{V96}. On the other hand,
starting from the seminal work of Caffarelli-Silvestre \cite{CS09} the theory of fractional Laplacian has significantly expanded in many directions and there is a large existing literature for this operator. 
The fractional Laplacian operators have been extensively used for 
mathematical modelling, for instance anomalous diffusion \cite{CR13,V12},
crystal dislocation \cite{DFV}, water waves \cite{BV16}. However, there are
other types of nonlocal operators that are also of importance. For instance,
relativistic operators appearing in quantum mechanics \cite{A17,FF14},
sum of fractional Laplacians of different order appearing in the modelling of
acoustic wave propagation in attenuating media \cite{ZH14}. This calls for
consideration of a general family of L\'evy operators (including the above mentioned nonlocal operators) for which a unified theory
can be developed. This motivates us to study positive solutions to
the following nonlocal logistic equation
\begin{equation*}
\begin{split}
\Psidel u &= au- f(x,u)- c h(x,u) \quad \text{in}\; D\,,
\\
u&=0 \quad \text{in}\; D^c\,,
\end{split}
\end{equation*}
where $a, c\in\RR$, $h$ represents the harvesting term
and $\Psi(-\Delta)$  denotes the generator of a subordinate Brownian motion and the subordinator is unique determined by its Laplace exponent
$\Psi$. For more details about $\Psidel$ please see \cref{S-Psi}. One of our main
goals in this article is to study existence and multiplicity of solutions for different values of $a$ and $c$. For $\Psidel=-\Delta$ similar problems have been
studied widely in literature (cf. \cite{CHS,CMS,CDT,DS06,GT09,KS,OSS,SS03}). But
for nonlocal situation there are only few results and 
to the best of our knowledge, all of them consider
the case $\Psidel=(-\Delta)^{\nicefrac{\alpha}{2}}$, the fractional Laplacian
(cf. \cite{BRR,CR13,CGH,MM20,RS14b}). Our results not only generalizes the existing works but also introduces several new methods. 
Recently, there have been quite a few works studying pde involving $\Psidel$ (cf. \cite{BL17a,BL17b,B18a,BS20,BGPR,KL21,KL20,KKL16,KKLL}).
 We also mention the recent work 
Biswas-L\H{o}rinczi \cite{BL19} where several maximum principles and generalized
eigenvalue problems for $\Psidel$ have been studied. Our novelty in this work also comes from the study of the long time asymptotic of the parabolic pde
\begin{equation*}
\begin{split}
(\partial_t-\Psidel) u + a u - f(x, u) &=\, 0 \quad \text{in}\; D\times [0, T),
\\
u(x, T)\,=\, u_0(x)\; \; \text{and}\; u(x, t)&=\, 0\quad \text{in}\;  D^c\times[0, T].
\end{split}
\end{equation*}
We use several potential theoretic tools to establish this long time behaviour.

Before we conclude this section let also also mention another type of nonlocal
kernel, known dispersal nonlocal kernel, widely used 
to model nonlocal reaction-diffusion equations ( cf. \cite{BCV,CDLL,GMR,HMMV} and
references therein). It should be noted that dispersal nonlocal kernels are quite different from the nonlocal kernels of $\Psidel$ and therefore, the proof techniques involved in these models are different from ours.

\subsection{A quick introduction to $\Psidel$}\label{S-Psi}
The class of non-local operators we would be interested in are generators of a large family of L\'evy processes,
known as subordinate Brownian motions. These processes are obtained by a time change of a Brownian motion  by independent 
subordinators.
In this section we briefly recall the essentials of the subordinate process which will be particularly used in this article.

A Bernstein function is a non-negative completely monotone function, that is, an element of the set
$$
\mathcal B = \left\{f \in \cC^\infty((0,\infty)): \, f \geq 0 \;\; \mbox{and} \:\; (-1)^n\frac{\D^n f}{\D x^n} \leq 0,
\; \mbox{for all $n \in \mathbb N$}\right\}.
$$
In particular, Bernstein functions are increasing and concave. We will consider the following subset
$$
{\mathcal B}_0 = \left\{f \in \mathcal B: \, \lim_{x\downarrow 0} f(x) = 0 \right\}.
$$
 For a detailed discussion of Bernstein functions we refer to the monograph \cite{SSV}.
Bernstein functions are closely related to subordinators. Recall that a subordinator $\{S_t\}_{t\geq 0}$
is a one-dimensional, non-decreasing L\'evy
process defined
on some
probability space $(\Omega_S, {\sF}_S, \mathbb P_S)$ .
 The Laplace transform of a subordinator is given by a
Bernstein function, i.e.,
\begin{equation*}
\label{lapla}
\Exp_{\mathbb P_S} [e^{-xS_t}] = e^{-t\Psi(x)}, \quad t, x \geq 0,
\end{equation*}
 where $\Psi \in {\mathcal B}_0$. In particular, there is a bijection between the set of subordinators on a given
probability space and Bernstein functions with vanishing right limits at zero.

Let $ B$ be an $\RR^d$-valued Brownian motion on the Wiener space 
$(\Omega_W,{\sF}_W, \mathbb P_W)$, running twice
as fast as standard $d$-dimensional Brownian motion, and let $ {S}$ be an independent subordinator with characteristic
exponent $\Psi$. The random process
$$
\Omega_W \times \Omega_S \ni (\omega_1,\omega_2) \mapsto B_{S_t(\omega_2)}(\omega_1) \in \Rd\,,
$$
is called subordinate Brownian motion under $ {S}$. For simplicity, we will denote a subordinate Brownian motion
by $ \{X_t\}_{t\geq 0}$, its probability measure for the process starting at $x \in \RR^d$ by $\mathbb P_x$, and expectation with respect
to this measure by $\Exp_x$. Note that the characteristic exponent of a pure jump process $\{X_t\}_{t\geq 0}$  is given
by
\begin{equation*}
\Psi(|z|^2)= \int_{\RR^d\setminus\{0\}} (1-\cos(y\cdot z)) j(|y|) \, \D{y},
\end{equation*}
where the L\'evy measure of $\{X_t\}_{t\geq 0}$ has a density $y\mapsto j(|y|)$, $j:(0, \infty)\to (0, \infty)$, with respect to
the Lebesgue measure, given by
\begin{equation}\label{E2.3}
j(r) = \int_0^\infty (4\pi t)^{-d/2} e^{-\frac{r^2}{4t}} \, \mathfrak m (\D{t}),
\end{equation}
where $\mathfrak m$ is the unique measure on $(0, \infty)$ satisfying \cite[Theorem~3.2]{SSV}
$$\Psi(s)=\int_{(0, \infty)} (1-e^{- s t}) \mathfrak m (\D{t}).$$
In particular, we have
$$\int_{\Rd} (|y|^2\wedge 1)\, j(|y|)\, \D{y}<\infty.$$
In this article we impose the following \textit{weak scaling condition} on the subordinators.
\begin{equation}\label{A1}\tag{A1}
\text{There are $0<\upkappa_1\leq \upkappa_2<1\leq b_1$ such that}\quad \frac{1}{b_1}\Big(\frac{R}{r}\Big)^{\upkappa_1}\leq \frac{\Psi(R)}{\Psi(r)}\leq b_1\Big(\frac{R}{r}\Big)^{\upkappa_2}\quad\text{for $1\leq r\leq R<\infty$},
\end{equation}
and,
\begin{equation}\label{A2}\tag{A2}
\text{there is $b_2>1$ such that}\quad j(r)\leq b_2\,j(r+1)\quad\text{for $r\geq 1$.}
\end{equation}
There is large family of subordinators that satisfy \eqref{A1} (see \cite{BL17b,KKLL}).
Moreover, any complete Bernstein function (see \cite[Definition~6.1]{SSV}) satisfying \eqref{A1} also satisfies 
\eqref{A2} (\cite[Theorem~13.3.5]{KSV}, \cite{KSV-12b}). The conditions \eqref{A1}-\eqref{A2} are imposed 
throughout this article without any further mention. It is also helpful to keep in mind that
for any $c>0$ we have
$$j(r)\asymp\frac{\Psi(r^{-2})}{r^d}\quad \text{for}\; 0< r<c\,,$$
where the comparison constants might depend on $c$ and whenever \eqref{A1} holds for all $R\geq r>0$
then we may take $c=\infty$  (see \cite{BGR14b}).

\begin{example}\label{Eg1.1}
Some important examples of complete Bernstein functions $\Psi$ satisfying \eqref{A1} are given by
\begin{itemize}
\item[(i)]
$\Psi(x)=x^{\alpha/2}, \, \alpha\in(0, 2]$, with ${\upkappa_1}=\upkappa_2 = \frac{\alpha}{2}$;
\item[(ii)]
$\Psi(x)=(x+m^{2/\alpha})^{\alpha/2}-m$, $m> 0$, $\alpha\in (0, 2)$, with 
${\upkappa_1}=\upkappa_2 =\frac{\alpha}{2}$;
\item[(iii)]
$\Psi(x)=x^{\alpha/2} + x^{\beta/2}, \, \alpha, \beta \in(0, 2]$, with ${\upkappa_1} = \frac{\alpha}{2}
\wedge \frac{\beta}{2}$,  and $\upkappa_2 = \frac{\alpha}{2} \vee \frac{\beta}{2}$;
\item[(iv)]
$\Psi(x)=x^{\alpha/2}(\log(1+x))^{-\beta/2}$, $\alpha \in (0,2]$, $\beta \in [0,\alpha)$ with
${\upkappa_1}=\frac{\alpha-\beta}{2}$ and $\upkappa_2=\frac{\alpha}{2}$;
\item[(v)]
$\Psi(x)=x^{\alpha/2}(\log(1+x))^{\beta/2}$, $\alpha \in (0,2)$, $\beta \in (0, 2-\alpha)$, with
${\upkappa_1}=\frac{\alpha}{2}$ and $\upkappa_2=\frac{\alpha+\beta}{2}$.
\end{itemize}
Corresponding to the examples above, the related processes are
(i) $\frac{\alpha}{2}$-stable subordinator, (ii) relativistic $\frac{\alpha}{2}$-stable subordinator, (iii) sums
of independent subordinators of different indices, etc.
\end{example}
The operator $-\Psidel$ is defined by
\begin{align}\label{opt}
-\Psidel f(x) &= \frac{1}{2}\int_{\Rd}\left(f(x+y)+f(x-y)-2f(x)\right)j(|y|)\,\D{y}
\\
&= \int_{\Rd} (f(x+y)-f(x)-\Ind_{\{|y|\leq 1\}} y\cdot \grad f(x)) j(|y|)\,\D{y}, \nonumber
\end{align}
which is classically defined for $f\in\cC^2_b(\Rd)$.
Here $\cC^2_b(\Rd)$ denotes the space of all bounded continuous function in $\Rd$ that are twice continuously differentiable.
 Also, $-\Psidel$ is the generator of the
strong Markov process $\{X_t\}_{t\geq 0}$ we introduced above. In connection to the examples above, the
related $-\Psidel$ operators are (i) $\frac{\alpha}{2}$-fractional Laplacian, (ii) $\frac{\alpha}{2}$-relativistic operator, (iii) sum of fractional Laplacians etc.

\subsection{Problem and main results}
Let $D\subset\Rd$ be a bounded $\cC^{1,1}$ domain. For positive constants $a, c$ we consider the 
following nonlocal logistic equation with a harvesting term
\begin{equation}\label{Prob}
\begin{split}
\Psidel u &= au- f(x,u)- c h(x,u) \quad \text{in}\; D\,,
\\
u &>0 \quad \text{in}\; D\,,
\\
u&=0 \quad \text{in}\; D^c\,,
\end{split}
\end{equation}
where $f:\bar{D}\times [0, \infty)\to [0, \infty), h:\bar{D}\times[0, \infty)\to [0, \infty)$ are given continuous functions satisfying
\begin{equation}\tag{A3}\label{A3}
\begin{gathered}
s\mapsto f(x, s),\, h(x, s)\; \mbox{are continuously differentiable}, f(x, 0)=f_s(x,0)=0,
\\
\frac{\D}{\D{s}}\left[\frac{f(x, s)}{s}\right]>0\; \mbox{for $s>0$},\; \lim_{s\to\infty} 
\inf_{x\in D}\frac{f(x,s)}{s}=\infty\,,\\
\mbox{and}\; h\; \mbox{is bounded with}\; \max_{\bar D} h(x, 0)>0\,.
\end{gathered}
\end{equation}
A typical example for $f$ is $b(x) u^2$ where $b$ in a positive continuous function.
 By a solution of \eqref{Prob} we mean viscosity solution. For a definition and regularity properties 
of viscosity solutions see \cref{S-prelim} below. As well known, existence of solutions to \eqref{Prob} is
closely connected with the principal eigenvalue of the operator $-\Psidel$. It is also known  that there 
are only countably many eigenvalues $0<\lambda_1<\lambda_2\leq\lambda_3\to\infty$
satisfying (see \cite{BL17a})
\begin{equation*}
-\Psidel\varphi_n + \lambda_n\varphi_n=0\quad \text{in}\; D, \quad \text{and}\quad \varphi_n=0\quad 
\text{in}\; D^c.
\end{equation*}
The first eigenvalue $\lambda_1$ is simple and $\varphi_1>0$ in $D$.
The principal eigenvalue $\lambda_1$ also satisfies a Berestycki-Nirenbarg-Varadhan \cite{BNV} type
characterization, that is, 
\begin{equation}\label{Eigen}
\lambda_1=\sup\{\lambda\; :\;\exists\;\psi\in\cC_{b, +}(D) \; \mbox{such that}\; -\Psidel\psi
+\lambda\psi\leq 0 \; \mbox{in}\; D\},
\end{equation}
where $\cC_{b, +}(D)$ denotes the collection of all bounded, non-negative continuous functions on $\Rd$ that
are positive inside $D$. Before we state our fist main result we recall the notion of stability for a solution
$u$ to the boundary value problem
\begin{equation}\label{E1.4}
\begin{split}
-\Psidel u + g(x, u) &=0 \quad \text{in}\; D\,,
\\
u&=0 \quad \text{in}\; D^c\,.
\end{split}
\end{equation}
A solution $u$ of \eqref{E1.4} is said to be a {\it stable solution} if the Dirichlet principal eigenvalue of the operator $-\Psidel + g_s(x, u)$ is positive, otherwise we say $u$ is an {\it unstable solution}.
Our first result is about the logistic equation (i.e., $h=0$)
\begin{theorem}\label{T1.1}
The logistic equation
\begin{equation}\label{E-Lg}
\begin{split}
\Psidel u &= au- f(x,u) \quad \text{in}\; D\,,
\\
u &>0 \quad \text{in}\; D\,,
\\
u&=0 \quad \text{in}\; D^c\,,
\end{split}
\end{equation}
has no solution for $a\leq \lambda_1$ and has exactly one solution $v_a$ for $a>\lambda_1$. Furthermore, 
the function $(\lambda_1, \infty)\ni a\mapsto v_a$ is continuous, increasing and $v_a$ is stable.
\end{theorem}
When $\Psidel=-\Delta$, \cref{T1.1} is well known. See for instance
Oruganti, Shi and Shivaji \cite[Theorem~2.5]{OSS}. For $\Psidel=(-\Delta)^{\nicefrac{\alpha}{2}}$
(i.e., the fractional Laplacian), similar result (without stability analysis of solutions)
 is obtained recently by Marinelli-Mugani 
\cite[Proposition~4.2]{MM20} using a variational technique
(see also Chhetri-Girg-Hollifield
\cite[Theorem~2.8]{CGH}). We also refer to the work of Berestycki, Roquejoffre and Rossi \cite[Theorem~1.2]{BRR}
which establishes a similar result for the fractional Laplacian for a periodic patch model in $\Rd$.
 We not only obtain uniqueness of solutions but also establish the result for 
a large class of L\'evy operators. It should also be noted that we work in the framework of 
viscosity solution and therefore, the standard variational technique (as used in \cite{BRR,OSS,CGH}) does
not work here. Also, our approach is quite robust in the sense that it can also be applied to 
non-translation invariant operators and non self-adjoint operators.

Next we consider the harvesting term $h$ and study existence of positive solutions. Note that we allow 
$h$ to depend on $u$. One such popular example is the predation function $h(x, s)=\frac{s}{1+s}$, although our approach does not cover this particular function.
The case $h(x, s)=h(x)$ is known as constant yield harvesting. Letting $F(x, u)= a u-f(x, u)-ch(x, u)$
in \eqref{Prob} we see that $F(x, 0)\leq 0$. Such problems are known as semipositone problems, see \cite{CHS,CMS,DS06,SS03} and references therein. When $\Psidel=-\Delta$, existence and multiplicity
of solutions to \eqref{Prob} have been widely studied; see for instance, Korman-Shi \cite{KS},
Oruganti-Shi-Shivaji \cite{OSS}, Costa-Dr\'{a}bek-Tehrani \cite{CDT}, Gir\~{a}o-Tehrani \cite{GT09}
and references therein.
We obtain the following bifurcation  result for equation \eqref{Prob}.
\begin{theorem}\label{T1.2}
Suppose that $a>\lambda_1$ and  $\inf_{s\in [0, K]} h(\cdot, s)\gneq 0$ in $D$ for every $K>0$.
Then the following hold.
\begin{itemize}
\item[(i)] There exists a positive constant $c_\circ$ such that \eqref{Prob} has a maximal
solution $u_1(x, c)$ for $c<c_\circ$.
\item[(ii)] There is no solution for $c>c_\circ$.
\item[(iii)] There exist positive $\delta, \tilde{c}$ such that for every $a\in (\lambda_1, \lambda_1+\delta)$
there exists a solution $u_2(x, c)$ to \eqref{Prob} for each $c\in (0,\tilde{c})$ and $u_2\lneq u_1$. Furthermore,
$\lim_{c\to 0+}\norm{u_2(\cdot, c)}_{\cC(D)}=0$.
\item[(iv)] There exists $\widehat{c}\in (0, \tilde{c})$ so that for any $a\in (\lambda_1, \lambda_1+\delta)$,
$u_1, u_2$ are the only solutions to \eqref{Prob} for $0<c\leq\widehat{c}$ .
\end{itemize}
\end{theorem}

\begin{remark}
The condition $\inf_{s\in [0, K]} h(\cdot, s)\gneq 0$ is used to prove nonexistence of solution for
large values of $c$. This condition does not have any influence on \cref{T1.2}(iii)
and (iv).
\end{remark}

The above result should be compared with \cite[Theorem~3.2 and 3.3]{OSS} which establish a similar result
for $\Psidel =-\Delta$ and $h(x, u)=h(x)$. To our best knowledge, there are no similar existing results for
nonlocal operators.
For the fractional Laplacian operators only existence of a solution is obtained for $c>0$ and
$a>\lambda_1$ in \cite[Theorem~2.9]{CGH}. The main idea in obtaining \cref{T1.2}(iii) is to apply
the implicit function theorem of Crandall and Rabinowitz \cite{CR71}. In case of the Laplacian
this is applied on the forward operator \cite[Theorem~3.3]{OSS}. But the same method can not applied
for nonlocal operators due to lack of appropriate Schauder estimates. We instead consider the inverse operator
(see \eqref{E2.3} below) and establish appropriate estimates so that the implicit function theorem can be applied.

As a corollary to the proof of \cref{T1.2} we get the following uniqueness result which generalizes \cite[Theorem~3.4]{OSS}. In the following result $V$ denotes the potential 
measure function of ladder-height process corresponding to $\{X^1_t\}$ (see \cref{S-prelim}).

\begin{corollary}\label{C1.1}
Suppose that
\begin{equation}\label{EC1.1A}
\sup_{s\in[0, k]}\sup_D\Bigl|\frac{h(x,s)}{V(\delta_D(x))}\Bigr|\,<\, \infty,
\end{equation}
for every finite $k$. Then for every $a>2\lambda_1$, there exists a $\breve{c}\in (0, c_\circ)$ so that for 
every $c\in (0, \breve{c})$, there exists a unique solution $u$ to \eqref{Prob} satisfying
\begin{equation}\label{EC1.1B}
\lambda_1 \, u(x)\geq c\, h(x, u(x)),\quad x\in \Rd.
\end{equation}
\end{corollary}

Next we discuss the long time behaviour of the parabolic nonlocal equation. Consider the terminal value problem
\begin{equation}\label{para}
\begin{split}
(\partial_t-\Psidel) u + a u - f(x, u) &=\, 0 \quad \text{in}\; D\times [0, T),
\\[2mm]
u(x, T)\,=\, u_0(x)\; \; \text{and}\; u(x, t)&=\, 0\quad \text{in}\;  D^c\times[0, T].
\end{split}
\end{equation}
By a solution of \eqref{para} we mean a potential theoretic solution. More precisely, we say
$u\in \cC(\Rd\times [0, T])$ is a solution to 
\begin{equation}\label{E1.6}
\begin{split}
(\partial_t-\Psidel) u + \ell(x, t) &=\, 0 \quad \text{in}\; D\times [0, T),
\\[2mm]
u(x, T)\,=\, g(x)\; \; \text{and}\; u(x, t) &=\, 0\quad \text{in}\;  D^c\times[0, T],
\end{split}
\end{equation}
if
\begin{equation}\label{E1.7}
u(x, t)=\Exp_x[g(X_{(T-t)\wedge\uptau})] + \Exp_x\left[\int_0^{(T-t)\wedge\uptau} \ell(X_s, t+s) ds\right],
\quad (x, t)\in D\times[0, T]\,,
\end{equation}
where $\uptau$ denotes the first exit time of $X$ from $D$.
It can be shown that potential theoretic solutions are same as viscosity solution
of \eqref{E1.6} (see \cref{L4.1} below). The benefit of working with \eqref{E1.7} is that it allows us to
make use of the underlying probabilistic structure of the model. Our next main result is the following

\begin{theorem}\label{T1.3}
Let $u_T$ be the positive and bounded solutions of \eqref{para} in $[0, T]$. Then the following hold.
\begin{itemize}
\item[(a)] For $a> \lambda_1$, we have $\lim_{T\to\infty} u_T(x, 0)\to v_a$, uniformly in $D$,
where $v_a$ is the unique solution of \cref{E-Lg}.
\item[(b)] For $a\leq \lambda_1$, we have $\lim_{T\to\infty} u_T(x, 0)\to 0$, uniformly in $D$.
\end{itemize}
\end{theorem}

To the best of our knowledge, there are no available results similar to \cref{T1.3}
in nonlocal setting. However, there are quite a few works on the fractional 
Fisher-KPP equation in $\Rd$; see for instance, Berestycki-Roquejoffre-Rossi \cite{BRR},
Cabr\'e-Roquejoffre \cite{CR13},
L\'eculier \cite{AL19} and references therein. For nonlocal dispersal
operators in $\Rd$ large time behaviour has been studied by 
Berestycki-Coville-Vo \cite{BCV}, Cao-Du-Li-Li \cite{CDLL},
 Su-Li-Lou-Yang \cite{SLLY} and references therein. The method used in these works are not applicable for our model. Since our nonlocal operator is quite general in nature there are no existing parabolic pde estimate (other than fractional Laplacian) that can be used to obtain our result. So we rely on the heat-kernel estimates of 
the underlying stochastic process $X$, and hence the reason to use probabilistic representation of the solution. 

The rest of the article is organized as follows: In \cref{S-prelim} we introduce
the relation between viscosity solution and the Green function representation.
We also gather few known results in this section which is used later in our proofs.
\cref{T1.1,T1.2} are proved in \cref{S-T1.1} whereas \cref{S-T1.3} contains the proof
of \cref{T1.3}.

\section{Preliminaries}\label{S-prelim}
In this section we recall the notion of nonlocal viscosity solutions, introduced by Caffarelli and Silvestre
\cite{CS09}, and its connection with potential theory. We also gather few results which will later be used 
to prove our main results. Denote by $\cC^2_b(x)$ the space of all bounded continuous functions in $\Rd$
that are twice continuously differentiable in some neighbourhood around $x$.

\begin{definition}\label{D2.1}
A function $u:\Rd\to\RR$, upper-semicontinuous in $\bar{D}$, is said to be a viscosity
subsolution of 
$\Psidel u = f$ in $D$,
if for every $x\in D$ and a test function $\xi\in\cC^2_b(x)$ satisfying $\xi(x)=u(x)$ and
$\xi(y)> u(y)$ for $y\in\Rd\setminus\{x\}$ we have $\Psidel \xi(x)\leq f(x)$.

We say $u$ is a viscosity super-solution of $\Psidel u = f$, 
if $-u$ is a viscosity subsolution of $\Psidel u = -f$ in $D$.
 Furthermore, $u$ is said to be a viscosity solution if it is both a viscosity sub- and super-solution.
\end{definition}
The viscosity solution of 
\begin{equation}\label{E2.1}
\Psidel u = f\quad \text{in}\; D,\quad \text{and}\quad u=0\quad \text{in}\; D^c\,,
\end{equation}
can be represented using Green function and this representation is going to play a key role 
in this article. We need few notations to introduce this representation. 
Let $\uptau$ be the first exit time of $X$ from $D$ i.e.,
$$\uptau=\inf\{t>0 \;: \; X_t\notin D\}.$$
We define the killed process $\{X^D_t\}$ by 
$$X^D_t=X_t \quad \text{if}\; t<\uptau, \quad \text{and}\quad X^D_t=\partial \quad \text{if}\; t\geq \uptau,$$
where $\partial$ denotes a cemetery point. $X^D_t$ has transition density $p_D(t, x, y)$ and its 
transition semigroup $\{P^D_t\}_{t\geq 0}$ is given by
\begin{equation}\label{E2.2}
P^D_t f(x) = \Exp_x[f(X_t)\Ind_{\{t<\uptau\}}]=\int_D f(y) p_D(t, x, y)\, \D{y}.
\end{equation}
The Green function of $X^D$ is defined by 
$$G^D(x, y)=\int_0^\infty p_D(t, x, y)\, \D{t}\,.$$
Then the solution of \eqref{E2.1} can be represented as (see \cite[Section~3.1]{B18},\cite{KKLL})
\begin{equation}\label{E2.3}
u(x)=\cG f(x) \df\int_D G^D(x, y) f(y) \, \D{y}=\Exp_x\left[\int_0^{\uptau} f(X_t) \, \D{t}\right],
\end{equation}
where the last equality follows from \eqref{E2.2}. 

For some of our proofs below we will use some information on the normalized ascending ladder-height process of
$\{X^1_t\}_{t\geq 0}$, 
where $X^{1}_t$ denotes the first coordinate of $X_t$. Recall that the ascending ladder-height
process of a L\'evy process $\{Z_t\}_{t\geq 0}$
 is the process of the right inverse $\{Z_{L^{-1}_t}\}_{t\geq 0}$, where $L_t$
is the local time of $Z_t$ reflected at its supremum (for details and further information we refer to 
\cite[Chapter~ 6]{B}).
Also, we note that the ladder-height process of $\{X^1_t\}_{t\geq 0}$ is a subordinator with Laplace exponent
$$
\tilde\Psi(x)=\exp\left(\frac{1}{\pi}\int_0^\infty \frac{\log \Psi(x^2y^2)}{1+y^2}\, \D{y}\right), \quad x \geq 0.
$$
Consider the potential measure $V(x)$ of this process on the half-line $(-\infty, x)$. Its Laplace transform is
given by
$$
\int_0^\infty V(x) e^{-sx}\, \D{x}= \frac{1}{s\tilde\Psi(s)}, \quad s > 0.
$$
It is also known that $V=0$ for $x\leq 0$, the function $V$ is continuous and strictly increasing in $(0, \infty)$
and $V(\infty)=\infty$ (see \cite{F74} for more details). As shown in \cite[Lemma~1.2]{BGR14} and 
\cite[Corollary~3]{BGR14b},
there exists a constant $C = C(d)$ such that
\begin{equation}\label{E2.4}
\frac{1}{C}\,{\Psi(r^{-2})}\leq \frac{1}{V^2(r)}\leq C\, {\Psi(r^{-2})}, \quad r>0.
\end{equation}
This function $V$ will appear in several places of this article. Let us recall the following 
up to the boundary regularity result from \cite[Theorem~1.1 and 1.2]{KKLL}

\begin{theorem}\label{T2.1}
Assume \eqref{A1}-\eqref{A2} and $f\in\cC(D)$. Let $u$ be the solution of \eqref{E2.1}.
Then for some constant $C$, dependent on $d, D, \Psi$,
we have
\begin{equation}\label{ET2.1A}
\norm{u}_{C^{\phi}(D)}\,\leq \, C \norm{f}_{L^\infty(D)},
\end{equation}
where $\phi=\Psi(r^{-2})^{-\frac{1}{2}}$ and
$$\norm{u}_{C^{\phi}(D)}\df \norm{u}_{\cC(D)} + \sup_{x, y\in D,x\neq y}\frac{|u(x)-u(y)}{\phi(|x-y|)}.$$
Furthermore, there exists $\alpha$, dependent on $d, D, \Psi$, satisfying
\begin{equation}\label{ET2.1B}
\Big|\Big|\frac{u}{V(\delta_D)}\Big|\Big|_{C^{\alpha}(D)}\,\leq \, C \norm{f}_{L^\infty(D)},
\end{equation}
where $\delta_D$ denotes the distance function from $\partial D$.
\end{theorem}
Using \eqref{A1}, $\phi(r)\leq \kappa r^{\upkappa_1}$ for $r\leq 1$, for some constant $\kappa$,
and thus, it follows from
\eqref{ET2.1A} that $u$ is $\upkappa_1$-H\"{o}lder continuous upto the boundary. \eqref{ET2.1B} provides a fine
boundary decay estimate and this should be compared with the results in \cite{RS14}.
Our next result is the Hopf's lemma which we borrow from \cite[Theorem~3.3]{BL19}.

\begin{theorem}\label{T2.2}
Let $u\in \cC_{b}(\Rd)$ be a non-negative viscosity solution of
\begin{equation*}
-\Psidel u + c(x) u\leq 0\quad \text{in}\; D,
\end{equation*}
where $c$ is a bounded function. Then either $u\equiv 0$ in $\Rd$ or $u>0$ in $D$. Furthermore,
if $u>0$ in $D$, then there exists $\eta>0$ satisfying
\begin{equation}\label{ET2.2A}
\frac{u(x)}{V(\delta_D(x))}\,>\, \eta \quad \text{for}\; x\in D\,.
\end{equation}
\end{theorem}
To introduce our next results we required the principal eigenvalue for the operator 
$-\Psidel + c$ where $c$ is a continuous and bounded function in $D$. The principal eigenvalue is defined
in the same fashion as in \cite{BNV} and given by
\begin{equation}\label{E2.8}
\lambda(c)\,=\,\sup\{\lambda\; :\;\exists\;\psi\in\cC_{b, +}(D) \; \mbox{such that}\; -\Psidel\psi
+(c(x)+\lambda)\psi\leq 0 \; \mbox{in}\; D\}.
\end{equation}
Note that for $c=0$ we have $\lambda(0)=\lambda_1$. Next we recall the following
refined maximum principle from \cite[Theorem~3.4 and Lemma~3.1]{BL19}.
\begin{theorem}\label{T2.3}
Suppose that $\lambda(c)>0$ and $v\in\cC_b(\Rd)$ be a solution to
$$-\Psidel v + c v\geq 0 \quad \text{in}\; D, \quad v\leq 0\quad \text{in}\; D^c.$$ 
Then we have $v\leq 0$. 

Again, if $w\in\cC_b(\Rd)$ is a solution to 
$$-\Psidel w + (c(x) +\lambda(c))w \geq 0 \quad \text{in}\; D, \quad w\leq 0\quad \text{in}\; D^c,
\quad w(x_0)>0,$$
for an $x_0\in D$, then $w=t \varphi^*$ for some $t>0$, where $\varphi^*$ denotes the positive principal 
eigenfunction corresponding to $\lambda(c)$.
\end{theorem}

The next result is an anti-maximum principle which is slightly stronger than \cite[Theorem~3.5]{BL19}.
\begin{theorem}\label{T2.4}
Let $f\in\cC(\bar D)$ and $f\lneq 0$. Then there exists a $\delta>0$ such that for every $\lambda\in(\lambda(c),
\lambda(c)+\delta)$ if $u$ is a solution of 
\begin{equation}\label{ET2.3A}
-\Psidel u + (c(x)+\lambda)u=f\quad \text{in}\; D, \quad \text{and}\quad u=0\quad \text{in}\; D^c\,,
\end{equation}
then $\sup_{D}\frac{u(x)}{V(\delta_D(x))}<0$. 
\end{theorem}

\begin{proof}
Using \cite[Theorem~3.5]{BL19} we have a $\delta_1>0$  such that for any $\lambda\in 
(\lambda(c), \lambda(c)+\delta_1)$ if $u$ is a solution to \eqref{ET2.3A} then $u<0$ in $D$.
Now suppose, on the contrary,  that the conclusion of the theorem does not hold. Then we find a sequence of $\delta_n\to 0$ and solution
$u_n<0$ satisfying
\begin{equation}\label{ET2.3B}
\max_{\partial{D}}\frac{u_n(x)}{V(\delta_D(x))}\,=\,0.
\end{equation}
First we observe that $\norm{u_n}_{L^\infty}\to\infty$ as $n\to\infty$. Otherwise, using the argument of Step 1
in \cite[Theorem~3.5]{BL19} we obtain a solution $u\lneq 0$ of 
$$-\Psidel u + (c(x)+\lambda^*) u\,=\, f(x)\quad \text{in}\; D, \quad u=0\quad \text{in}\; D^c.$$
In view of \cref{T2.3}, we must have $u=t\varphi^*$ for some $t<0$,
 where $\varphi^*$ is the positive
Dirichlet principal eigenfunction of $-\Psidel + c$ in $D$.
This is not possible since
$f\neq 0$. Thus we must have $\norm{u_n}_{L^\infty}\to\infty$. Define $v_n=\frac{u_n}{\norm{u_n}_{L^\infty}}$. Then the argument of
Step 2 in \cite[Theorem~3.5]{BL19} gives us
$$\max_{\bar D}\Big|\frac{v_n}{V(\delta_D(x))}-\frac{t\varphi^*}{V(\delta_D(x))}\Big|\to 0, 
\quad \text{as}\; n\to\infty,$$
for some $t<0$. Combining with \eqref{ET2.3B} we must find a point 
$x_0\in\partial{D}$ such that
$\frac{\varphi^*(x_0)}{V(\delta_D(x_0))}=0$. But $\frac{\varphi^*}{V(\delta_D)}$ can be continuously
extended in $\bar{D}$ (by \cref{T2.1}) and the extension is positive in $\bar{D}$, by \cref{T2.2}. 
Thus we arrive at a contradiction. Hence we have a $\delta>0$ as claimed by the theorem.
\end{proof}
Before we conclude this section let us also mention the following implicit function theorem from 
\cite[Appendix]{CR71}. In the following theorem $\cX$, $\cY$ denote Banach spaces.

\begin{theorem}\label{T2.5}
Let $(s_0, u_0)\in \RR\times\cX$ and $F:\RR\times\cX\to \cY$ be continuously differentiable in
some some neighbourhood of $(s_0, u_0)$. Assume that $F(s_0, u_0)=0$. Suppose that $F_u(s_0, u_0)$ is 
a linear homeomorphism of $\cX$ onto $\cY$. Then there is exactly one $\cC^1$ function 
$z:(s_0-\varepsilon, s_0+\varepsilon)\to \cX$ with $z(s_0)=0$ satisfying $F(s, u_0 + z(s))=0$
for $s\in (s_0-\varepsilon, s_0+\varepsilon)$ where $\varepsilon$ is some positive number.
\end{theorem}

\section{Proof of \cref{T1.1,T1.2}}\label{S-T1.1}
The goal of this section is to prove \cref{T1.1,T1.2}. Let us start with the main comparison principle
required in this section.

\begin{lemma}\label{L3.1}
Suppose that $g :\bar{D}\times[0, \infty) \to \RR$ is a continuous function, 
locally Lipschitz in the second variable uniformly with respect to the first,
such that
$$\frac{g(x, s)}{s}\; \mbox{is strictly decreasing for $s > 0$}$$
at each $x \in D$. In addition, also assume that $g(x, 0)=0$ and $g_s(x, 0)$
is continuous in $\bar{D}$.
Let  $u,v \in \cC_b(\Rd)$ be such that
\begin{enumerate}
\item $-\Psidel v + g(x,v) \leq 0 \leq\, \tilde{g}(x)= - \Psidel u + g(x,u)$ in $D$, where $\tilde{g}$ is a
continuous function.
\item $v > 0$, $u\gneq 0$ in $D$ and $v \geq u = 0$ in $D^c$.
\end{enumerate}
Then we have $v \geq u$ in $\Rd$.
\end{lemma}

\begin{proof}
Let $\varrho = \sup \{ t : \; tu < v\; \text{in}\; D\}$.
 Clearly, $\varrho<\infty$. Also, $\varrho>0$. Note that
by Hopf's lemma, \cref{T2.2}, we have
$$\inf_{D}\frac{v(x)}{V(\delta_D(x))}\geq \eta>0,$$
and by \eqref{ET2.1B} 
\begin{equation}\label{EL3.1A}
\sup_{x\in D} \Big| \frac{u(x)}{V(\delta_D(x))}\Big|\leq \eta_1
\end{equation}
for some $\eta_1>0$. Thus for some small $t_0>0$ we would have $v>t_0 u$ in $D$, giving us $\varrho\geq t_0>0$.
To complete the proof it is enough to show that $\varrho\geq 1$. On the contrary, we suppose that $\varrho<1$.
Let $w=\varrho u$.
Since $ \frac{g(x, s)}{s}$ is strictly decreasing for $s > 0$ we have
\begin{align}\label{EL3.1B}
- \Psidel w + g(x,w) &=  -\Psidel w + \frac{g(x, \varrho u)}{\varrho} \varrho\nonumber
 \\
&\gneq  \varrho[ - \Psidel u + g(x, u)] \geq 0 \quad \text{in}\; D\,,
\end{align}
Applying \cite[Lemma~5.8]{CS09} we then have
$$
- \Psidel (v-w) + g(x,v) - g(x,w) \leq 0\quad \text{in}\; D,
$$
which in turn, gives 
$$
 - \Psidel (v-w) +\left(\frac{g(x,v) - g(x,w)}{v-w}\right) (v-w) \leq 0\quad \text{in}\; D.
$$
Applying Hopf's lemma, \cref{T2.2}, we have either $v-w =0$ in $\Rd$ or  
$\inf_{D}\frac{v(x)-w(x)}{V(\delta_D(x))} > \eta$. The first option is not possible due to \eqref{EL3.1B}.
Again, if the second option holds, then using \eqref{EL3.1A} we can find $t_1>0$ satisfying
$u-w> t_1 u$ in $D$ implying $v> (\varrho+t_1) u$ in $D$. This contradicts the definition of $\varrho$.
Hence we must have $\varrho\geq 1$.
\end{proof}
Now we are ready to prove \cref{T1.1}.
\begin{proof}[{\bf Proof of \cref{T1.1}}]
Recall that  $(\lambda_1, \varphi_1)$ is the Dirichlet principal eigenpair, that is,
\begin{equation}\label{ET1.1A}
\begin{split}
- \Psidel \varphi_1 + \lambda_1 \varphi_1 &= 0 \; \text{in} \; D,
\\
\varphi_1 &= 0 \; \text{in} \; D^c\,.
\end{split}
\end{equation}
Suppose that $a < \lambda_1$ and $v$ is a positive solution of \eqref{E-Lg}. Then
$$
- \Psidel v + av = \frac{f(x,v)}{v}v \geq 0\quad \text{in}\; D\,
$$
since $\frac{f(x,s)}{s} \geq 0$ for $s \geq 0$. Applying the refined maximum principle \cref{T2.3}
we get $v\leq 0$ in $\Rd$ which is a contradiction.

Similarly, if $v$ is a positive solution with $a=\lambda_1$,
we obtain $- \Psidel v + \lambda_1 v = \frac{f(x,v)}{v}v \geq 0$ in $D$. 
Applying second part of \cref{T2.3}
we have $v = t\varphi_1$ for some $t>0$ which would imply
$$
- \Psidel \varphi_1 + \lambda_1 \varphi_1 - t^{-1}f(x, t\varphi_1) = 0\quad \text{in}\; D, 
$$
giving us
$$
-f(x , t \varphi_1) = 0\quad \text{in}\; D\,.
$$
This is not possible since $t\varphi_1> 0$ in $D$. Thus we have established that no positive solution is possible
for $a\leq \lambda_1$.

Next we consider the case where $a>\lambda_1$. Existence of solution would be proved using a standard monotone 
iteration method. To do so we need to construct a subsolution and supersolution. Let $\underline{u}=k\varphi_1$ where $k\in (0, 1)$. Then we obtain from \eqref{ET1.1A} that
\begin{align*}
- \Psidel\underline{u} + a\underline{u} - f(x, \underline{u}) 
&= (a - \lambda_1) \underline{u} - f(x, \underline{u}) 
\\
&= \underline{u} \left( (a- \lambda_1) - \frac{f(x, k\varphi_1)}{k \varphi_1} \right)\quad \text{in}\; D\,.
\end{align*}
Since by mean value theorem $\frac{f(x,q)}{q} =  f_s(x,r)$ for some $r \in (0,q)$ and $f_s(x, 0)=0$, by
choosing $k$ small we would easily have 
$$\left( (a- \lambda_1) - \frac{f(x, k\varphi_1)}{k \varphi_1} \right)\,>\,0\quad  \text{in}\; D.$$
Thus we obtain a subsolution $\underline{u}$. Again, since 
$$\lim_{s\rightarrow \infty} \inf_{x \in D} \frac{f(x,s)}{s} = \infty\,, $$
there exist large $M > \norm{\underline{u}}_{\cC(D)}$ satisfying
$\frac{f(x,M)}{M} \geq a $ for all $x$ in $D$.
Fixing $v = M$ we get
\begin{align*}
- \Psidel v + av - f(x,v) &\leq  0 \quad \text{in}\; D\,.
\end{align*}
Thus $v$ is a super-solution. Now the existence
of a solution is standard using monotone iteration method. Let us just
sketch the argument. Define $H(x, u)=au-f(x, u)$ and let $ \theta > 0$ be a Lipschitz constant for $H(x, \cdot)$ on the interval $[ 0, M ]$, i.e.,
$$
\vert H(x, q_1) - H(x, q_2) \vert \leq \theta \vert q_1  - q_2 \vert \quad
 \text{for} \; q_1, q_2 \in  \left[ 0 , M \right], \;  x \in D\,.
$$
Now consider the solutions of the following family of problems:
\begin{align*}
-\Psidel u^{n+1} - \theta u^{n+1}  &=  - H(x,u^n)  - \theta u^n  \quad x \in D,
\\
u^{n+1} &= 0 \quad x \in D^c\,,
\end{align*}
with $u^0=\underline{u}$. It is standard to check that $u^0\leq u^1\leq u^2\leq \cdots\leq v$. Applying
\cref{T2.1} and Arzel\`{a}-Ascoli thereom it can be shown that the sequenece converges uniformly in $\Rd$ to
a limit $v_a\geq \underline{u}$ and $v_a$ is a viscosity solution to \eqref{E-Lg}. See \cite[Lemma~3.3]{B18a}
for more details. Uniqueness of solution to \eqref{E-Lg} follows from \cref{L3.1}.

Next we prove stability of the solution $v_a$ for $a> \lambda_1$. Note that given $a_2\geq a_1>\lambda_1$ we have
$$\Psidel v_{a_2} \geq a_1 v_{a_2}-f(x, v_{a_2})\quad \text{in}\; D\,.$$
Therefore, by \cref{L3.1}, we have $v_{a_1}\leq v_{a_2}$. Again, due to \cref{T2.1}, it can easily be shown that $a\mapsto v_a$ is continuous.

Fix $a> \lambda_1$ and define $w=(1+h)v_a$ for $h>0$. Since
$$(1+h) f(x, s)< f(x, (1+h)s)\quad \text{for}\; s\geq 0,\; x\in D\,,$$
we have
$$-\Psidel w + a w - f(x, w)\leq 0 \quad \text{in}\; D.$$
Using \cite[Lemma~5.8]{CS09} we then obtain
\begin{align*} 
-\Psidel (h v_a) + a (hv_a) - f(x, w)+f(x, v_a) &= -\Psidel(w-v_a) - a(w-v_a) - f(x, w) + f(x, v_a)
\\
&\leq 0 \quad \text{in}\; D.
\end{align*}
Dividing by $h$ on both sides we get
$$-\Psidel v_a + a v_a - \left[\frac{f(x, w)-f(x, v_a)}{h v_a}\right] v_a\leq 0\quad \text{in}\; D.$$
Letting $h\to 0$ and using the stability property of viscosity solutions \cite[Lemma~4.5]{CS09} we 
obtain
$$-\Psidel v_a + a v_a - f_s(x, v_a) v_a\leq 0\quad \text{in}\; D.$$
Then it follows from \eqref{E2.8} that the principal eigenvalue $\lambda^*$ of the operator 
$-\Psidel + a- f_s(x, v_a)$ is non-negative. Now suppose $\lambda^*=0$. Then from the proof of 
\cite[Theorem~3.2]{BL19} (see the last part of the proof) we get that $v_a$ is a principal eigenfunction i.e.,
$$-\Psidel v_a + a v_a - f_s(x, v_a) v_a= 0\quad \text{in}\; D.$$
Combining with \eqref{E-Lg} we have $f_s(x, v_a) v_a= f(x, v_a)$ for all $x\in D$. But by \eqref{A3}
we have $sf_s(x, s)-f(x, s)>0$ for all $s>0$. Thus we have a contradiction, giving us $\lambda^*>0$.
This completes the proof.
\end{proof}
For the remaining part of this section we  consider the equation with the harvesting term $h$:
\begin{equation}\label{E3.4}
\begin{split}
\Psidel u &= au- f(x,u)- c h(x,u) \quad \text{in}\; D\,,
\\
u &>0 \quad \text{in}\; D\,,
\\
u&=0 \quad \text{in}\; D^c\,,
\end{split}
\end{equation}
where $h$ satisfies the conditions in \eqref{A3}. We start with the following lemma about non-existence.

\begin{lemma}\label{L3.2}
The following hold.
\begin{enumerate}
\item[(i)] If $a \leq \lambda_1$ and $c \geq 0$ then equation \eqref{E3.4} has no non negative solution except $u=0$ when $c=0$.
\item[(ii)] Suppose that $\inf_{s\in [0, K]} h(\cdot, s)\gneq 0$ for any $K>0$.
Then for $a > \lambda_1$, there exists $M>0$ such that  equation \eqref{E3.4} has no 
nonzero non-negative solution when $c>M$.
\end{enumerate}
\end{lemma}

\begin{proof}
First we consider (i). Note that 
$$-\Psidel u + a u = f(x,u)+ c h(x,u)\geq 0 \quad \text{in}\; D.$$
Then the arguments of \cref{T1.1} shows that there is no non-negative $u$ satisfying above equation when
$a\leq \lambda_1$.

(ii) Fix $a > \lambda_1$. We will prove theorem by contradiction. Assume that there exists positive 
increasing sequence $c_n \rightarrow \infty$ and solution  $u_n\gneq 0$ to \eqref{E3.4}.
We claim that for any non-negative solution $u$ to \eqref{E3.4}, we have
\begin{equation}\label{EL3.2A}
\Vert u \Vert_{L^{\infty}} \leq K,
\end{equation} 
for some $K$ and all $c\geq 0$.
Since $\lim_{s\rightarrow \infty} \inf_{x \in D} \frac{f(x,s)}{s} = \infty$ 
there exist large $K > 0$ such  that 
$\frac{f(x,K)}{K} \geq a $ for all x in $D$.
Taking $v = K$ we get
\begin{align*}
- \Psidel v + av - f(x,v) &= a K - f(x,K)
\\
&= K\left( a- \frac{f(x,K)}{K} \right) \leq 0\quad \text{in}\; D\,.
\end{align*}
So $v$ is a super-solution. Thus
$$ - \Psidel v + av - f(x,v) \leq  0 \leq ch(x,u) = - \Psidel u + au - f(x,u)\quad \text{in}\; D\,.
$$
Using \cref{L3.1} we obtain \eqref{EL3.2A}.
Now dividing both sides of \eqref{E3.4} by $c_n$ we have
\begin{align*}
&- \Psidel \left( \frac{u_n}{c_n} \right) + a \frac{u_n}{c_n} - \frac{f(x,u_n)}{c_n} + \min_{s\in [0, K]}h(x,s)
\\
&\leq - \Psidel \left( \frac{u_n}{c_n} \right) + a \frac{u_n}{c_n} - \frac{f(x,u_n)}{c_n} + h(x,u_n) = 0
\quad \text{in}\; D\,.
\end{align*}
Since $\frac{u_n}{c_n}$ and $\frac{f(x,u_n)}{c_n} $ converges to $0$ as $c_n \rightarrow \infty$ we get from above that
$\min_{s\in [0, K]}h(x,s)=0$ which is a contradiction. Hence the result.
\end{proof}

Next we prove existence of solution for small values of $c$.
\begin{lemma}\label{L3.3}
Fix $a > \lambda_1$. Then there exist $c_1$ such that for $c\in (0,c_1)$ equation \eqref{E3.4} has a  solution
$u$ satisfying $u\geq m\beta\varphi_1$ where $m, \beta$ are independent of $c\in (0, c_1)$.
\end{lemma}

\begin{proof}
We will prove existence of a positive solution using a monotone iteration method.
Let $v$ be the unique solution of 
\begin{align*}
 \Psidel v  &= 1 \quad \text{in} \; D\,,
 \\
v &= 0 \quad \text{in} \; D^c.
\end{align*}
From maximum principle it is evident that $v>0$ in $D$.
Also, recall the principal eigenfunction $\varphi_1$ from \eqref{ET1.1A}.
Using \cref{T2.1,T2.2} we obtain that
\begin{equation}\label{EL3.3A}
\Big|\frac{v(x)}{V(\delta_D(x))}\Big|\leq \eta_1,
\quad \eta_2\leq \frac{\varphi_1(x)}{V(\delta_D(x))}\leq \eta_3, \; x\in D\,,
\quad \eta_1,\, \eta_2,\, \eta_3>0\,.
\end{equation}
 Thus
\begin{align*}
\varphi_1(x) \geq \frac{\eta_2}{\eta_1} v(x)\quad x\in D\,.
\end{align*}
Taking $\varepsilon(\beta) = (1- \beta) \frac{\eta_2}{\eta_1}$ we get
$\varphi_1 - \varepsilon v \geq \beta \varphi_1$.
Define $\phi = m (\varphi_1 - \varepsilon v)$. Note that $\phi \geq m \beta \varphi_1$. Now
\begin{align*}
- \Psidel \phi + a \phi - f(x,\phi) -c h(x, \phi) 
&= - \lambda_1 m \varphi_1 + m \varepsilon + a \phi - f(x,\phi) -ch(x, \phi)
\\
&\geq - \frac{\lambda_1}{\beta} \phi  + a \phi - f(x,\phi) + m \varepsilon -c
\norm{h}_{L^\infty}
\\
&\geq \left( a - \frac{ \lambda_1}{\beta} - \frac{f(x, \phi)}{\phi} \right) \phi +m\varepsilon - c\norm{h}_{L^\infty}\,.
\end{align*}
Now choose $ \beta\in(\frac{\lambda_1}{a}, 1)$ and then choose $m$ small so that 
$$\frac{f(x, \phi)}{\phi} \leq a - \frac{ \lambda_1}{\beta}\quad \text{in}\; D .$$ 
Then for any 
$c \leq m \norm{h}_{L^\infty}^{-1}\varepsilon = \norm{h}_{L^\infty}^{-1} (1 -\beta)\frac{\eta_2}{\eta_1} m\df c_1 $ we have 
$$
-\Psidel \phi + a \phi - f(x,\phi) -c h(x,\phi) \geq 0\quad \text{in}\; D\,.
$$
Thus we have a  subsolution for all $c \leq c_1$. Again, as shown in \cref{L3.2}, we can choose a
$K$ to serve a supersolution. Then using a standard monotone iteration method (same as in \cref{T1.1})
we can obtain a solution $u$ to \eqref{E3.4} satisfying $u\geq \phi$.
\end{proof}

Using \cref{L3.2,L3.3} we obtain the following.
\begin{theorem}\label{T3.1}
Assume the setting of \cref{T1.2}.
Suppose that $a > \lambda_1$. Then there exists $c_\circ \geq c_1$ such that 
\begin{enumerate}
\item[(i)] for $0 < c < c_\circ$, \eqref{E3.4}  has a maximal positive solution $u_1 (x, c)$ such that for any solution $v(x, c)$ of \eqref{E3.4} we have $u_1 \geq v$. Furthermore,
\begin{equation}\label{ET3.1A}
\lim_{c\to 0+}\norm{u_1(\cdot, c)-v_a}_{\cC(D)}=0;
\end{equation}
\item[(ii)] for $c > c_\circ$, \eqref{E3.4} has no positive solution.
\end{enumerate}
\end{theorem}

\begin{proof}
(i)\;
From \cref{T1.1}, we know that \eqref{E-Lg} has a unique positive solution 
$v_a$ when $a > \lambda_1$. Let $u$ be any nonnegative solution of \eqref{E3.4}. Then 
$$
-\Psidel v_a + a v_a - f(x, v_a) = 0 < ch(x,u) = -\Psidel u + au - f(x,u)\quad \text{in}\; D\,.
$$
Since $u = v_a =0$ in $D^c$, using \cref{L3.1} we have that $u \leq v_a$ in $\Rd$.
Thus whenever \eqref{E3.4} has a nonnegative solution for some $c$, we can construct maximal solution of 
$u_1 (\cdot , c)$ for the same parameter $c$ as follows:  we take $v_a$ as a supersolution of \eqref{E3.4}, any solution $u$ as a subsolution, and start the monotone iteration sequence starting from $v_a$. Then we obtain a solution $u_1$ in between $v_a$ and $u$; in particular, $u_1 \geq u$. Since u can be any solution, the limit of the iterated sequence starting from $v_a$ is the maximal solution. 
$$ 
c_\circ = \sup\{ c > 0 : \eqref{E3.4} \text{ has a solution with this}\, c \}. $$
From \cref{L3.3} it is clear that $c_\circ\geq c_1$. Now we show that for any 
$c \in (0 , c_\circ)$, \eqref{E3.4} has a solution. Then from previous argument we can construct maximal solution for any $c \in (0, c_\circ)$. Fix $c \in (0 , c_\circ)$. By definition of $c_\circ$
we can find $c' > c$ such that  \eqref{E3.4} has a solution $u$ for $c'$. This also implies
\begin{align*}
- \Psidel u + au - f(x,u) - c h(x,u) & \geq 0 \quad \text{in}\; D\,.
\end{align*}
Since $v_a\geq u$ using a monotone iteration argument we can find a solution for the parameter $c$.
Now to show \eqref{ET3.1A} we observe from \cref{L3.3} that for $c\in (0, c_1)$
\begin{equation}\label{ET3.2B}
m \beta \varphi_1\leq u_1(x, c)\leq v_a\quad \text{in}\; D.
\end{equation}
Applying \cref{T2.1} we see that the family $\{u_1(\cdot, c)\}_{c\leq c_1}$ is equicontinuous and any limit point
$\xi\in\cC(\Rd)$ ,as $c\to 0+$, would solve 
$$\Psidel \xi = a\xi- f(x,\xi)\quad \text{in}\; D.$$
From \eqref{ET3.2B} it follows that $\xi>0$ in $D$. Thus, by \cref{T1.1}, $\xi=v_a$. This gives us \eqref{ET3.1A}.

(ii) follows from the definition of $c_\circ$.
\end{proof}

Now we complete the proof of \cref{T1.2}. By $\cC_0(D)$ we denote the space of
all continuous functions in $\bar{D}$ vanishing on the boundary.

\begin{proof}[{\bf Proof of \cref{T1.2}}]
(i) and (ii) follows from \cref{T3.1}. So we consider (iii).
The main idea of this proof is to use \cref{T2.5} but due to lack of appropriate Schauder type estimate we can
not apply the theorem on the forward operator.
Recall the Green operator $\cG$ associated to the Dirichlet problem \eqref{E2.1},
that is,
\begin{equation}\label{E-Gr}
\cG f(x) \df\int_D G^D(x, y) f(y) \, \D{y}=\Exp_x\left[\int_0^{\uptau} f(X_t) \D{t}\right].
\end{equation}
In view of \eqref{ET2.1A}, $\cG:\cC_0(D)\to \cC_0(D)$ is a compact, bounded linear operator.
Now extend $h$ on $\bar{D}\times\RR$ by defining
$h(x, s)=h(x, 0)+sh_s(x, 0)$. Then $s\mapsto h(x, s)$ is $\cC^1$. 
We define $F:\RR\times \cC_0(D)\to \cC_0(D)$ by
$$F(c, u)=\cG(a u - f(x,u)-ch(x, u) ) -u.$$
Since $\cG$ is linear, it is clear that $F$ is continuously differentiable in a neighbourhood of $(0, 0)$.
In particular,
$$DF(c, u)(c_1, w) = \cG(aw-f_s(x, u)w-c_1 h(x, u)-ch_s(x,u)w)-w.$$
Also, $F(0, 0)=0$.
Define $Tw\df\, F_u(0, 0) w= \cG(a w)-w$. It is clear that $T$ is a bounded linear operator. Furthermore,
$Tw=0$ implies $\cG(aw)=w$ giving us
$w\in \cC_0(D)$ and $-\Psidel w + aw=0$. Since $a$ is not an eigenvalue, 
we must have $w=0$. Thus $T$ is injective. Since $\cG$ is  compact,
by Fredholm alternative on Banach spaces $T$ is also surjective and $T^{-1}$ is also bounded linear.
Therefore, we can apply the implicit theorem \cref{T2.5}
to obtain  a $\cC^1$ curve $(c, z(c))$ in $(-\varepsilon, \varepsilon)$,
 with $z(0)=0$ and $F(c, z(c))=0$. In other words,
\begin{equation}\label{ET1.23A}
\begin{split}
\Psidel z(c) &= az(c)- f(x,z(c))- c h(x,z(c)) \quad \text{in}\; D\,,
\\
z(c) &=0 \quad \text{in}\; D^c\,.
\end{split}
\end{equation}
To complete the proof we only need to show that there exists $\tilde{c}$ such that $z(c)>0$ in $D$.
Considering $c=0$ and $f(x)=-h(x, 0)$ in \cref{T2.4} we choose the corresponding $\delta$ from 
\cref{T2.4}. Fix $a\in (\lambda_1, \lambda_1+\delta)$. Since $c\mapsto z(c)$ is $\cC^1$ we have
$\frac{1}{|c|}\norm{z(c)}_{\cC(D)}\leq K$ for some $K$ and all small $c$. Defining $U_c=\frac{z_c}{c}$ we 
obtain from \eqref{ET1.23A} that
\begin{equation}\label{ET1.23B}
\begin{split}
\Psidel U_c &= a U_c- F_c(x) U_c-  h(x,z(c)) \quad \text{in}\; D\,,
\\
U_c &=0 \quad \text{in}\; D^c\,,
\end{split}
\end{equation}
where $F_c(x)=\frac{f(x, z(c))}{z(c)}$. Note that the rhs of \eqref{ET1.23B} is uniformly bounded for
all $c$ small. Thus applying \cref{T2.1} we find that $\{U_c\}, \{\frac{U_c}{V(\delta_D)}\}$ are uniformly
H\"{o}lder continuous in $D$. In particular, the sequences are pre-compact. Now suppose that there exists $c_n\to 0$ such that $z(c_n)\ngtr 0$ in $D$. Then we can extract a subsequence $n_k$ satisfying
\begin{equation}\label{ET1.23C}
\sup_{x\in D} \Big|\frac{U_{c_{n_k}}(x)}{V(\delta_D(x))}-\frac{W(x)}{V(\delta_D(x))}\Big|\to 0, 
\text{as}\; n_k\to 0,
\end{equation}
for some $W\in\cC_0(D)$.  Furthermore, from the stability property of viscosity solution 
\cite[Corollary~4.7]{CS09} we obtain
$$\Psidel W = a W-  h(x,0) \quad \text{in}\; D\,,\quad W=0\quad \text{in}\; D^c.$$
Using \cref{T2.4} we have $W>0$ in $D$ and $\inf_D \frac{W}{V(\delta_D)}>0$. From \eqref{ET1.23C} we then
have $U_{c_{n_k}}>0$ in $D$ for all large $n_k$ which contradicts the fact $z(c_n)\ngtr 0$ in $D$ for all $n$.
Hence we can find $\tilde{c}>0$ such that $u_2(c)\df z(c)>0$ in $D$. Moreover, 
$$\lim_{c\to 0+}\norm{u_2(c)}_{\cC(D)}=0.$$

(iv) Suppose, on the contrary, that there exists a sequence $c_n\to 0$ and solutions $v(\cdot,c_n)$ of \eqref{E3.4}
corresponding to $c_n$ and $v(\cdot,c_n)\neq u_1(\cdot, c_n)$ and $v(\cdot,c_n)\neq u_2(\cdot, c_n)$. To simplify the
notation we denote by $v^n=v(\cdot,c_n), u^n_1=u_1(\cdot,c_n), u_2^n=u_2(\cdot, c_n)$. Since, by \cref{T2.1}, $\{v^n\}$ is equi-conitnuous, from \cref{T1.1} one of the following hold.
\begin{itemize}
\item[(a)] There exists a subsequence $\{n_k\}$ satisfying $\norm{v^{n_k}-v_a}_{\cC(D)}=0$, as
$n_k\to\infty$.
\item[(b)] There exists a subsequence $\{n_k\}$ satisfying $\norm{v^{n_k}}_{\cC(D)}=0$, as
$n_k\to\infty$.
\end{itemize}
We arrive a contradiction below in each of the cases. Consider (a) first. Since 
$u^n_1$ is the maximal solution we have $v^n\leq u^n_1\leq v_a$. Thus, by \cref{T3.1}, we have
$$\lim_{n_k\to\infty}\, \norm{u^{n_k}_1-v^{n_k}}_{\cC(D)}=0.$$
Defining $w^{n}=u^{n}_1-v^{n}$ and using \eqref{E3.4} we get
\begin{equation}\label{ET1.2D}
\Psidel w^{n_k}=a w^{n_k} - \frac{f(x, u^{n_k}_1)-f(x, v^{n_k})}{w^{n_k}} w^{n_k}
- c_{n_k} \frac{h(x, u^{n_k}_1)-h(x, v^{n_k})}{w^{n_k}} w^{n_k}\quad \text{in}\; D\,.
\end{equation} 
Since $w^{n_k}\gneq 0$ in $D$, by \cref{T2.2}, we have $w^{n_k}>0$ in $D$. Normalize
$w^{n_k}$ by defining $\xi^{n_k}=\frac{1}{\norm{w^{n_k}}_{\cC(D)}}w^{n_k}$. From \eqref{ET1.2D}
we then have
\begin{equation}\label{ET1.2E}
\begin{split}
\Psidel \xi^{n_k} &=a \xi^{n_k} - \frac{f(x, u^{n_k}_1)-f(x, v^{n_k})}{w^{n_k}} \xi^{n_k}
- c_{n_k} \frac{h(x, u^{n_k}_1)-h(x, v^{n_k})}{w^{n_k}} \xi^{n_k}\quad \text{in}\; D\,,
\\
\xi^{n_k}&=0\quad \text{in}\; D^c\,,
\\
\xi^{n_k}&>0\quad \text{in}\; D\,.
\end{split}
\end{equation}
Using \cref{T2.1}, we see that $\{\xi^{n_k}\}$ is equicontinuous and then
 passing to the limit 
along some subsequence and using stability property of the viscosity solution in \eqref{ET1.2E}, we find a solution
$\xi\in\cC(\Rd)$ with $\xi>0$ in $D$ (due to \cref{T2.2}) satisfying
\begin{equation*}
\begin{split}
\Psidel \xi &=a \xi - f_u(x, v_a) \xi \quad \text{in}\; D\,,
\\
\xi&=0\quad \text{in}\; D^c\,,
\\
\xi&>0\quad \text{in}\; D\,.
\end{split}
\end{equation*}
But this contradicts the fact $v_a$ is a stable solution (see \cref{T1.1}). Thus (a) is not possible.
So we consider (b). Defining $w^n=u^n_2-v^n\neq 0$ and $\xi^n=\frac{1}{\norm{w^n}_{\cC(D)}} w^n$ and 
repeating a similar argument as above, we get a non-zero $\xi$ satisfying
\begin{equation*}
\Psidel \xi =a \xi  \quad \text{in}\; D\,,\quad 
\xi=0\quad \text{in}\; D^c\,,
\end{equation*}
which is a contradiction since $a$ is not an eigenvalue of $\Psidel$. Thus (b) is also not possible.
This completes the proof of the theorem.
\end{proof}

Next we establish \cref{C1.1}
\begin{proof}[{\bf Proof of \cref{C1.1}}]
First we show existence. Recall from 
\cref{L3.3} and \cref{T1.2}(i) that for any $c<c_1$ there exists a maximal solution
$u_1(c)=u_1(\cdot, c)$ of 
\begin{equation}\label{EC1.1C}
\begin{split}
\Psidel u &= au- f(x,u)- c h(x,u) \quad \text{in}\; D\,,
\\
u &>0 \quad \text{in}\; D\,,
\\
u&=0 \quad \text{in}\; D^c\,,
\end{split}
\end{equation}
satisfying $m \beta\varphi_1\leq u_1(c)\leq v_a$. Using \cref{T2.1} we see that 
$\{u_1(c)\}_{c<c_1}$, $\{\frac{u_1(c)}{V(\delta_D)}\}_{c<c_1}$ are equi-continuous 
family of positive functions. Since any subsequential limit of $\{u_1(c)\}_{c<c_1}$
, as $c\to 0$, would be a positive solution to \eqref{E-Lg} (by stability property 
of viscosity solutions), from \cref{T1.1} we obtain that
\begin{equation}\label{EC1.1D}
\lim_{c\to 0+}\,
\sup_D\,\Big|\frac{u_1(x,c)}{V(\delta_D)}-\frac{v_a(x)}{V(\delta_D)}\Big|=0.
\end{equation}
Since $\inf_D \frac{v_a(x)}{V(\delta_D(x))}>0$ by \cref{T2.2}, using \eqref{EC1.1D}
and \eqref{EC1.1A} we can find $c_2>0$ so that for every $c\in (0, c_2)$ we have
$$\lambda_1 u_1(c)\geq c h(x, u_1), \quad \text{in}\; x\in \Rd.$$

Next we show uniqueness. We claim that if
$w(c)=w(\cdot, c)$ be any solution to \eqref{EC1.1C} satisfying \eqref{EC1.1B}, then there 
exists $c_3>0$ and $\delta>0$ satisfying
\begin{equation}\label{EC1.1E}
\inf_{c\in (0, c_3)}\sup_{\Rd} w(c)>0\,.
\end{equation}
If this does not hold, then for a sequence $\{c_n\}, c_n\to 0,$ we would have
$\sup_{\Rd} w(c_n)= \sup_{D} w(c_n)\to 0$ as $n\to\infty$. 
Denote by $2\eta=a-2\lambda_1$. Using \eqref{A3}, we obtain
$$\eta w(c_n)- \frac{f(x, w(c_n))}{w(c_n)} w(c_n)>0, \quad \text{in}\; D\,,$$
for all large $n$. Hence, using \eqref{EC1.1B}, we obtain for all large $n$ that
\begin{align*}
-\Psidel w(c_n) + (\lambda_1+\eta) w(c_n)
&\leq -\Psidel w(c_n) + (\eta+2\lambda_1) w(c_n) - c_n h(x, w(c_n))
\\
&\leq -\Psidel w(c_n) + a w(c_n) - f(x, w(c_n))-c_n h(x, w(c_n))=0
\quad \text{in}\; D.
\end{align*}
But this contradicts the definition of $\lambda_1$ in \eqref{Eigen}. This gives us
\eqref{EC1.1E}. This also confirms that $\lim_{c\to 0} w(c)=v_a$. Then uniqueness follows from the argument of \cref{T1.2}(iv) (see situation (a) there).
\end{proof}

\section{Proof of \cref{T1.3}}\label{S-T1.3}
We prove \cref{T1.3} in this section. Also, we recall the subordinate Brownian motion $\{X\}$,
introduced in \cref{S-Psi}, defined on a complete filtered probability space
$(\Omega, \sF, \{\sF_t\}, \Prob)$ and $\{X_t, \sF_t\}$ is also a strong Markov process.
Most of the tools used in this section are probabilistic nature because
we use the potential theoretic interpretation of the solution.  Recall that
$u\in \cC(\Rd\times [0, T])$ is said to be a solution to 
\begin{equation}\label{E4.1}
\begin{split}
(\partial_t-\Psidel) u + \ell(x, t) &=\, 0 \quad \text{in}\; D\times [0, T),
\\
u(x, T)\,=\, g(x)\; \; \text{and}\; u(x, t)&=\, 0\quad \text{in}\;  D^c\times[0, T],
\end{split}
\end{equation}
if
\begin{equation}\label{E4.2}
u(x, t)=\Exp_x[g(X_{T-t})\Ind_{\{T-t<\uptau\}}] + \Exp_x\left[\int_0^{(T-t)\wedge\uptau} \ell(X_s, t+s) ds\right],
\quad (x, t)\in D\times[0, T]\,,
\end{equation}
where $\uptau$ denotes the first exit time of $X$ from $D$. Throughout this section
we assume that $g\in\cC_0(D)$.
 It is important to observe that \eqref{E4.2} is not
different from a viscosity solution. We recall the definition of viscosity solution. By
$\cC^{2, 1}_b(x, t)$ we denote the space of all bounded continuous functions in $\Rd\times[0, T]$ that are
in $\cC^{2, 1}$ class in some neighbourhood of $(x, t)$. The following definition of viscosity solution
can be found in \cite{CLD,Sil}.
\begin{definition}
An upper (lower) semicontinuous function $u$ is said to be a viscosity subsolution (supersolution)
 of \eqref{E4.1} if for every $(x, t)\in D\times [0, T)$ and $\varphi\in \cC^{2, 1}_b(x, t)$ satisfying
$$\varphi(x, t)=u(x, t),\quad \varphi(y, s)\geq u(y, s) \quad \text{for}\; y\in \Rd, \; t\leq s< t+\delta,$$

$$\left(\varphi(x, t)=u(x, t),\quad \varphi(y, s)\leq u(y, s) \quad \text{for}\; y\in \Rd, \; t\leq s< t+\delta, respectively,\right)$$
for some $\delta>0$, we have
$$ (\partial_t-\Psidel)\varphi(x, t)+\ell(x, t)\geq 0,$$
$$\left((\partial_t-\Psidel)\varphi(x, t)+\ell(x, t)\leq 0, \; respectively\right).$$
\end{definition}
The time derivative $\partial_t$ can also be replaced by the derivative in parabolic topology i.e.,
$$\partial_{t^+}\varphi(x, t)=\lim_{h\to 0+}\frac{\varphi(x, t+h)-\varphi(x, t)}{h}.$$
Let us first show that potential theoretic solution is also a viscosity solution.

\begin{lemma}\label{L4.1}
Let $u\in\cC_b(\Rd\times[0, T])$ satisfy \eqref{E4.2}. Assume that $\ell, g$ are continuous.
Then $u$ is the unique viscosity solution of 
\eqref{E4.1}.
\end{lemma}

\begin{proof}
Let $x\in\sB\subset D$. By $\uptau_\sB$ we denotes the exit time from $\sB$ i.e.,
$$\uptau_\sB=\inf\{t>0\; :\; X_t\notin \sB\}.$$
It is evident that $\uptau_\sB\leq\uptau$. First we show that for any $\delta< T-t$
\begin{equation}\label{EL4.1A}
u(x, t)=\Exp_x[u(X_{\delta\wedge\uptau_\sB}, t+\delta\wedge\uptau_\sB)]
 + \Exp_x\left[\int_0^{\delta\wedge\uptau_\sB} \ell(X_s, t+s) ds\right].
\end{equation}
Using \eqref{E4.2} we write
\begin{align*}
u(x, t) &=\Exp_x[g(X_{T-t})\Ind_{\{T-t<\uptau\}}] + 
\Exp_x\left[\int_0^{T-t} \ell(X_s, t+s)\Ind_{\{s<\uptau\}} ds\right]
\\
&=  \Exp_x[g(X_{T-t}) \Ind_{\{\delta\wedge\uptau_\sB\leq \uptau\}} \Ind_{\{(T-t)<\uptau\}}]
+ \Exp_x\left[\Ind_{\{\delta\wedge\uptau_\sB\leq \uptau\}}\int_{\delta\wedge\uptau_\sB}^{T-t} \ell(X_s, t+s)\Ind_{\{s<\uptau\}} ds\right]
\\
&\quad + \Exp_x\left[\Ind_{\{\delta\wedge\uptau_\sB\leq \uptau\}}\int_0^{\delta\wedge\uptau_\sB} \ell(X_s, t+s)\Ind_{\{s<\uptau\}} ds\right]
\\
&= \Exp_x\left[\Ind_{\{\delta\wedge\uptau_\sB\leq\uptau\}} \Exp_{X_{\delta\wedge\uptau_\sB}}
\left[g(X_{(T-t-\delta\wedge\uptau_\sB)}) \Ind_{\{T-t-\delta\wedge\uptau_\sB<\uptau\}}\right]\right]
\\
&\qquad + \Exp_x\left[\Ind_{\{\delta\wedge\uptau_\sB\leq \uptau\}}
\Exp_{X_{\delta\wedge\uptau_\sB}}\left[\int_0^{(T-t-\delta\wedge\uptau_\sB)} 
\ell(X_s, t+\delta\wedge\uptau_\sB+s)\Ind_{\{s<\uptau\}} ds\right]\right]
\\
&\qquad + \Exp_x\left[\Ind_{\{\delta\wedge\uptau_\sB\leq\uptau\}}\int_0^{\delta\wedge\uptau_\sB} \ell(X_s, t+s) ds\right]
\\
&= \Exp_x\left[\Ind_{\{\delta\wedge\uptau_\sB\leq \uptau\}}
u(X_{\delta\wedge\uptau_\sB}, t+\delta\wedge\uptau_\sB)\right] + \Exp_x\left[\Ind_{\{\delta\wedge\uptau_\sB\leq \uptau\}}\int_0^{\delta\wedge\uptau_\sB} \ell(X_s, t+s) ds\right]
\\
&= \Exp_x\left[u(X_{\delta\wedge\uptau_\sB}, t+\delta\wedge\uptau_\sB)\right] +
 \Exp_x\left[\int_0^{\delta\wedge\uptau_\sB} \ell(X_s, t+s) ds\right],
\end{align*}
where in the third line we use strong Markov property and in the last line we use 
the fact that $\Prob_x(\delta\wedge\uptau_\sB\leq\uptau)=1$ . This proves \eqref{EL4.1A}. This relation
is key to show that $u$ is also a viscosity solution. We only check that $u$ is a viscosity subsolution and 
the other part would be analogous. Consider $(x, t)\in D\times[0, T)$ and $\varphi\in \cC^{2, 1}_b(x, t)$ satisfying
$$\varphi(x, t)=u(x, t),\quad \varphi(y, s)\geq u(y, s) \quad \text{for}\; y\in \Rd, \; t\leq s< t+\delta.$$
Choose a ball $\sB$, centered at $x$, small enough so that $\varphi$ is $\cC^{2,1}$ in 
$\bar\sB\times[t, t+\delta]$. Let $\delta_1<\delta$.
Then applying Dynkin-It\^{o} formula we know that
\begin{align*}
\Exp_x\left[\int_0^{\delta_1\wedge\uptau_\sB} (\partial_t-\Psidel)\varphi(X_s, t+s)\D{s}\right]
&=\Exp_x[\varphi(X_{\delta_1\wedge\uptau_\sB}, t+\delta_1\wedge\uptau_\sB)]-\varphi(x, t)
\\
&\geq \Exp_x[u(X_{\delta_1\wedge\uptau_\sB}, t+\delta_1\wedge\uptau_\sB)]-u(x, t)
\\
&=- \Exp_x\left[\int_0^{\delta_1\wedge\uptau_\sB} \ell(X_s, t+s) ds\right],
\end{align*}
using \eqref{EL4.1A}. Since $\Prob_x(\uptau_\sB>0)=1$,
dividing both sides by $\delta_1$ and letting $\delta_1\to 0$ to obtain
$$(\partial_t-\Psidel)\varphi(x, t)+\ell(x, t)\geq 0.$$
Similarly, we can show that $u$ is a supersolution.

The uniqueness part follows using a similar argument as in \cite[Lemma~3.3]{Sil} (Note that the proof of
\cite[Lemma~3.2]{Sil} is based on the ideas from \cite{CS09} which works for general nonlocal operators).
\end{proof}

Our next lemma concerns representation of Schr\"{o}dinger equation.
\begin{lemma}\label{L4.2}
Suppose that $\ell, V$ are continuous and bounded in $D$ and $g\in C_0(D)$. Define
$$\varphi(x, t) =
\Exp_x\left[e^{\int_0^{(T-t)\wedge\uptau} V(X_s, t + s)\, \D{s}}g(X_{(T-t)\wedge\uptau})\right] + \Exp_x\left[\int_0^{(T-t)\wedge\uptau} e^{\int_0^{s} V(X_k, t + k)\, \D{k}} \ell(X_s, t+s)\, \D{s}\right].$$
Then $\varphi$ solves 
\begin{equation}\label{EL4.2A}
\begin{split}
(\partial_t-\Psidel)\varphi + \ell + V\varphi &=\, 0 \quad \text{in}\; D\times [0, T),
\\
\varphi(x, T)\,=\, g(x)\; & \text{and}\; \varphi(x, t)=\, 0\quad \text{in}\;  D^c\times[0, T].
\end{split}
\end{equation}
\end{lemma}

\begin{proof}
It is routine to check $\varphi$ is continuous (cf. \cite[Lemma~3.1]{BL17a}) and $\varphi(\cdot, t)=0$ in $D^c$. It also follows from the definition $\varphi(x, T)=g(x)$.
Now fix any $t\in [0, T)$ and $\delta<T-t$. Since $g$ and $\varphi$ vanish outside $D$ we obtain 
that
\begin{align}\label{EL4.2B}
&\varphi(x, t) \nonumber
\\
&=\Exp_x\left[\Ind_{\{T-t<\uptau\}}e^{\int_0^{T-t} V(X_s, t + s)\, \D{s}}g(X_{T-t})\right] + \Exp_x\left[\int_0^{(T-t)} \Ind_{\{s<\uptau\}} e^{\int_0^{s} V(X_k, t + k)\, \D{k}} \ell(X_s, t+s) \D{s}\right]\nonumber
\\
&= \Exp_x\left[\Ind_{\{\delta<\uptau\}} e^{\int_0^{\delta} V(X_s, t + s) \D{s}}
\Exp_{X_\delta}\left[\Ind_{\{T-t-\delta<\uptau\}}e^{\int_0^{T-t-\delta} V(X_s, t+\delta + s) \D{s}}g(X_{T-t-\delta})\right]\right]\nonumber
\\
&\quad + \Exp_x\left[\int_0^{\delta} \Ind_{\{s<\uptau\}} e^{\int_0^{s} V(X_k, t + k) \D{k}} \ell(X_s, t+s) \D{s}\right]
\nonumber
\\
&\quad + \Exp_x\left[\Ind_{\{\delta<\uptau\}}e^{\int_0^{\delta} V(X_s, t + s) \D{s}}\Exp_{X_\delta}\left[\int_0^{T-t-\delta} \Ind_{\{s<\uptau\}} e^{\int_0^{s} V(X_k, t+\delta + k)\, \D{k}} \ell(X_s, t+s) \D{s}\right]\right]\nonumber
\\
& = \Exp_x\left[\Ind_{\{\delta <\uptau\}} e^{\int_0^{\delta} V(X_s, t + s)\, \D{s}}
\varphi(X_\delta, t+\delta)\right] + \Exp_x\left[\int_0^{\delta} \Ind_{\{s<\uptau\}} e^{\int_0^{s} V(X_k, t + k)\, \D{k}} \ell(X_s, t+s) \D{s}\right],
\end{align}
where the second equality follows from the strong Markov property.
Now fix $x\in D$ and define
$$\xi(p)= \Exp_x[\varphi(X_{p\wedge\uptau}, t+p\wedge\uptau)].$$
Then, using \eqref{EL4.2B} we note that
\begin{align*}
&\xi(p)-\xi(p-\delta)
\\
& = \Exp_x[\varphi(X_{p\wedge\uptau}, t+p\wedge\uptau)]-
\Exp_{x}\left[\Ind_{\{p-\delta<\uptau\}}\Exp_{X_{p-\delta}}\left[\Ind_{\{\delta<\uptau\}} e^{\int_0^{\delta} V(X_s, t+p-\delta + s) \D{s}}
\varphi(X_\delta, t+p)\right]\right]
\\
&\qquad - \Exp_{x}\left[\Ind_{\{p-\delta<\uptau\}}\Exp_{X_{p-\delta}}\left[\int_0^\delta \Ind_{\{s<\uptau\}} e^{\int_0^{s} V(X_k, t+p-\delta + k) \D{k}} \ell(X_s, t+p-\delta+s) \D{s}\right]\right]
\\
&=\Exp_x[\varphi(X_{p\wedge\uptau}, t+p\wedge\uptau)]-
\Exp_{x}\left[\Ind_{\{p-\delta<\uptau\}}\Exp_{x}\left[\Ind_{\{p<\uptau\}} e^{\int_0^{\delta} V(X_{s+p-\delta}, t+p-\delta + s) \D{s}}
\varphi(X_{p}, t+p)\Big|\sF_{\uptau\wedge(p-\delta)}\right]\right]
\\
&\qquad - \Exp_{x}\left[\Ind_{\{p-\delta<\uptau\}}\Exp_{x}\left[\int_0^\delta \Ind_{\{s+p-\delta<\uptau\}} e^{\int_0^{s} V(X_{p-\delta+k}, t+p-\delta + k) \D{k}} \ell(X_{p-\delta+s}, t+p-\delta+s) \D{s} \Big|\sF_{\uptau\wedge(p-\delta)}\right]\right]
\\
&=\Exp_x[\varphi(X_{p}, t+p)\Ind_{\{p<\uptau\}}]-
\Exp_{x}\left[\Ind_{\{p<\uptau\}} e^{\int_0^{\delta} V(X_{s+p-\delta}, t+p-\delta + s) \D{s}}
\varphi(X_{p}, t+p)\right]
\\
&\qquad - \Exp_{x}\left[\int_0^\delta \Ind_{\{\uptau>s+p-\delta\}} e^{\int_0^{s} V(X_{p-\delta+k}, t+p-\delta + k) \D{k}} \ell(X_{p-\delta+s}, t+p-\delta+s) \D{s} \right]\,.
\end{align*}
Then, using the quasi-continuity property of $X$, we obtain
\begin{align*}
\lim_{\delta\to 0+} \frac{1}{\delta}(\xi(p)-\xi(p-\delta))
&= -\Exp_x[V(X_p, t+p)\varphi(X_{p}, t+p)\Ind_{\{\uptau>p\}}] - 
\Exp_x[\Ind_{\{\uptau\geq p\}} \ell (X_p, t+p)]
\\
&= -\Exp_x[V(X_p, t+p)\varphi(X_{p}, t+p)\Ind_{\{\uptau>p\}}] - 
\Exp_x[\Ind_{\{\uptau> p\}} \ell (X_p, t+p)]\df -\zeta(p),
\end{align*}
where the last line follows since $\Prob_x(\uptau=p)=0$. Hence the left derivative of 
$\xi$ exists and given by $\zeta$ which is continuous. Therefore, $\xi$ is a $\cC^1$ function. Now  using the fundamental theorem of calculus we obtain
\begin{align*}
\varphi(x, t)&= \xi(T-t) + \int_0^{T-t}\zeta(s)\, \D{s}
\\
&=\Exp_x[\varphi(X_{T-t})\Ind_{\{T-t<\uptau\}}] +
\Exp_s\left[\int_0^{T-t} (V(X_s, t+s)\varphi(X_{s}, t+s)+ \ell (X_s, t+s))\Ind_{\{s<\uptau\}}\D{s}\right]\,.
\end{align*}
Thus $\varphi$ solves \eqref{EL4.2A}.
\end{proof}
Next we get a parabolic comparison principle. Let $q:\bar{D}\times[0, \infty)\to [0, \infty)$ be
a continuous function, $\cC^1$ in its second variable and $q_s:\bar{D}\times[0, \infty)\to\RR$
is also continuous. Also, assume that $q(x, 0)=0$ and
$$s\mapsto \frac{q(x, s)}{s}\quad \mbox{is decreasing}.$$
\begin{lemma}\label{L4.3}
Let $u, v$ be two positive solutions of 
$$(\partial_t-\Psidel)w + q(x,w)=0 \quad \text{in}\; D\times[0, T), \quad w=0\; \text{in}\; D^c\times[0, T].$$
If $u(x, T)\leq v(x, T)$ in $\Rd$, then we also have $u\leq v$ in $\Rd\times[0, T]$.
\end{lemma}

\begin{proof}
Let $G(x, t)=\frac{q(x, u(x, t))}{u(x, t)}$ and $H(x, t)=\frac{q(x,v(x, t))}{v(x, t)}$. Then using 
\cref{L4.2} we obtain
\begin{align}
u(x, t) &= \Exp_x\left[e^{\int_0^{T-t} G(X_s, t + s) \D{s}}u(X_{T-t}, T)\Ind_{\{T-t<\uptau\}}\right],
\quad (x, t)\in D\times[0, T],\label{EL4.3A0}
\\
v(x, t) &= \Exp_x\left[e^{\int_0^{T-t} H(X_s, t + s) \D{s}}v(X_{T-t}, T)
\Ind_{\{T-t<\uptau\}}\right],\quad  (x, t)\in D\times[0, T].\label{EL4.3A00}
\end{align}
Note that without loss of generality we may assume $u(\cdot, T)\gneq 0$, otherwise from above we get
$u=0$ and then, there is nothing to prove. Let $K=\max_{\bar{D}\times[0, T]}(|G|+|H|)$. Then it 
is evident from above that 
\begin{equation}\label{EL4.3A}
v(x, t) \geq e^{-K T} u(x, t)\quad \text{for all}\; x, t.
\end{equation}
Define
$$\beta=\sup\{t\; :\; tu\leq v\quad \text{in}\; D\times[0, T]\}.$$
Using \eqref{EL4.3A} we get that $\beta\geq e^{-K T}$. To complete the proof we need to show that
$\beta\geq 1$. Suppose, on the contrary, that $\beta<1$. Denote by $u_1=\beta u$. Then for $w=v-u_1\geq 0$ we have
\begin{align*}
w(x, t) & = \Exp_{x}\left[w(X_{T-t}, T)\Ind_{\{T-t<\uptau\}}\right] + 
\Exp_x\left[\int_0^{(T-t)\wedge\uptau} \Bigl(q(X_s, v(X_s, t+s)) - \beta q(X_s, u(X_s, t+s))\Bigr) \D{s}\right]\,.
\end{align*}
For any $\delta\in (0, T-t)$, we can repeat the calculation of \eqref{EL4.2B} with $V=0$
to arrive at
$$w(x, t) = \Exp_{x}\left[\Ind_{\{\delta<\uptau\}} w(X_{\delta}, t+\delta)\right]
+\Exp_x\left[\int_0^\delta \Ind_{\{s<\uptau\}}\Big(q(X_s, v(X_s, t+s)) - \beta q(X_s, u(X_s, t+s))\Big)\, \D{s}\right].$$
By our assumption on $q$, $q(x, v)-\beta q(x, u)\geq q(x, v)- q(x, \beta u) \geq - M w$,
for some constant $M$. Thus defining 
$\xi(s)=\Exp_x[\Ind_{\{\uptau>s\}} w(X_{s}, t+s)]$
we obtain
$$\xi(\delta) \leq \xi(0) + M \int_0^\delta \xi(s) \D{s}.$$
Applying Gronwall's inequality we then have
$$\xi(T-t)\leq C w(x, t),$$
for some constant $C$, independent of $(x,t)\in D\times [0, T]$.
Since $w(x, T)\gneq 0$, we must have $w(x, t)>0$ for all $t< T$. Furthermore,
$w(x, T)\geq (1-\beta) u(x, T)$, implying
$$C w(x, t)\geq \xi(T-t)\geq (1-\beta) \Exp_x[\Ind_{\{T-t<\uptau\}} u(X_{s}, T)],$$
which combined with \eqref{EL4.3A0} gives
$\kappa u(x, t)\leq w(x, t)$ for $(x, t)\in D\times[0, T]$ and for some $\kappa>0$.
This certainly contradicts the definition of $\beta$. Hence $\beta\geq 1$, completing the proof.
\end{proof}

Next we establish a regularity property in space up to the boundary.
\begin{lemma}\label{L4.4}
Suppose that $g, \ell$ be such that $\norm{g}_{L^\infty}, \norm{\ell}_\infty\leq K$. Then 
for any $u$ satisfying
$$u(x, t)=\Exp_x[g(X_{T-t})\Ind_{\{T-t<\uptau\}}] + \Exp_x\left[\int_0^{T-t}\ell(X_s, t+ s)
\Ind_{\{s<\uptau\}} \D{s}\right],\quad (x, t)\in D\times[0, T],$$
we have, for $t<T$
$$|u(x,t)-u(y,t)|\leq C\, V(|x-y|), \quad x, y\in D,$$
for a constant $C$ dependent on $t, T, K$ where $V$ is the potential measure introduced in \cref{S-prelim}.
We can also choose the constant $C$ uniformly in $t$ varying in a compact subset of $[0, T)$.
\end{lemma}

\begin{proof}
Denote by
\begin{align*}
\sR_1(x) &= \Exp_x[g(X_{T-t})\Ind_{\{T-t<\uptau\}}]=\int_D g(y)p_D(T-t, x, z)\, \D{z},
\\[2mm]
\sR_2(x) &= \Exp_x\left[\int_0^{T-t}\ell(X_s, t+ s) \Ind_{\{s<\uptau\}} \D{s}\right]
= \int_0^{T-t} \ell(x, t+s) p_D(s, x, z)\, \D{z},
\end{align*}
where $p_d(t, x, z)$ denotes the transition density of the killed process $X^D$ (see \eqref{E2.2}).
Then from the arguments of \cite[Proposition~3.5]{KKLL} (see (3.13) of that paper) one can find a constant $C_1$, independent of $t, T$, satisfying 
\begin{equation}\label{EL4.4A}
|\sR_2(x)-\sR_2(y)|\leq C_1 V(|x-y|)\quad x, y\in D\,.
\end{equation}
To calculate $\sR_1$ we recall the following result from \cite[Theorem~1.1]{KR18}
and \cite[Theorem~1.3]{GKK} (see also \cite{CKS},\cite[Theorem~3.1]{KKLL})
\begin{align}
|\grad_x p_D(t, x, y)| &\leq C_2 \left(\frac{1}{\delta_D(x)\wedge 1}\vee \frac{1}{V^{-1}(\sqrt{t})}\right)
p_D(t, x, y) \quad x, y\in D,\label{EL4.4B}
\\[2mm]
p_D(t, x, y)&\leq C_3 \left(1\wedge\frac{V(\delta_D(x))}{\sqrt{t}}\right) \left(1\wedge\frac{V(\delta_D(y))}{\sqrt{t}}\right)
p(t, |x-y|)\quad x, y\in D\,,\label{EL4.4C}
\end{align}
for $t\in (0, T]$ and some constants $C_2, C_3$, dependent on $T$, where $p$ denotes the transition density of $X$. 
Using \eqref{A1} and \eqref{E2.4} we also have, for any $\kappa>0$, that
\begin{equation}\label{EL4.4D}
C_4^{-1}\left(\frac{R}{r}\right)^{\upkappa_1}\leq \frac{V(R)}{V(r)}\leq C_4\left(\frac{R}{r}\right)^{\upkappa_2} \quad 0<r\leq R\leq \kappa\,,
\end{equation}
where $C_4$ depends on $\kappa$.
Take $x, y\in D$. Suppose $2|x-y|\leq \max\{\delta_D(x), \delta_D(y)\}$. With no loss of generality
we may assume that $y\in B(x, \frac{1}{2}\delta_D(x))$. Note that for any point $z$ on the line 
joining $x$ and $y$ we get from \eqref{EL4.4B}-\eqref{EL4.4C}
\begin{align*}
|x-y||\grad_x p_D(T-t, z, y)|&\leq C_2 |x-y|\left(\frac{1}{\delta_D(z)\wedge 1}\vee 
\frac{1}{V^{-1}(\sqrt{T-t})}\right) p_D(T-t, z, y)
\\
&\leq C_2\, |x-y|\left(\frac{1}{\delta_D(z)\wedge 1}\vee \frac{1}{V^{-1}(\sqrt{T-t})}\right)
\\
&\quad \cdot C_3 \left(1\wedge\frac{V(\delta_D(z))}{\sqrt{T-t}}\right) 
\left(1\wedge\frac{V(\delta_D(y))}{\sqrt{T-t}}\right)p(T-t, |z-y|)
\\
&\leq C_5 \frac{|x-y|}{\delta_D(z)} V(\delta_D(z)),
\end{align*}
for some $C_5>0$.
Since $|x-y|\leq \frac{1}{2}\delta_D(x)\leq \delta_D(z)$,
using \eqref{EL4.4D} with $\kappa=\diam(D)$ we then obtain
$$\frac{|x-y|}{V(|x-y|)}|\grad_x p_D(T-t, z, y)\leq C_5 \frac{|x-y|}{\delta_D(z)} 
\frac{V(\delta_D(z))}{V(|x-y|)}
\leq C_4 C_5 \left(\frac{|x-y|}{\delta_D(z)}\right)^{1-\upkappa_2}\leq C_4 C_5.$$
Thus 
\begin{equation}\label{EL4.4E}
|\sR_1(x)-\sR_1(y)|\leq \int_D |g(y)| |p_D(T-t, x, z)-p_D(T-t, y, z)|\, \D{z}
\leq C_4 C_5 V(|x-y|) \norm{g}_{L^\infty} |D|.
\end{equation}
Now we consider the situation $2|x-y|\geq \max\{\delta_D(x), \delta_D(y)\}$. Then
using \eqref{EL4.4C}-\eqref{EL4.4D}
\begin{align}\label{EL4.4F}
|\sR_1(x)-\sR_1(y)|\leq C_7 (V(\delta_D(x))+V(\delta_D(y)))\leq 
C_7 (V(2|x-y|)+V(2|x-y|))\leq C_8 V(|x-y|),
\end{align}
for some constants $C_7, C_8$. Combining \eqref{EL4.4A}, \eqref{EL4.4E} and \eqref{EL4.4F} we get
the result.
\end{proof}

Now we are ready to prove an existence result.
\begin{lemma}\label{L4.5}
Let $q$ be same as in \cref{L4.3}. Also, assume that $q(x, \cdot)$ is $\cC^1$
, uniformly with respect to $x$. Let $\ell_1:\bar{D}\to (-\infty, 0]$,
$\ell_2:\bar{D}\to [0, \infty)$ be be two continuous functions.
Let $v_i, i=1,2,$ be a non-negative solution satisfying
$$(\partial_t -\Psidel) v_i + q(x, v_i) + \ell_i(x)=0\quad \text{in}\; D\times[0, T), \quad v(x, t)=0\; \text{in}\; D^c\times[0, T],$$
and let $g$ be such that $v_1(x, T)\leq g(x)\leq  v_2(x, T)$. 
Then there exists a unique solution $u\, (v_1\leq u\leq v_2)$ to 
\begin{equation}\label{EL4.5A}
\begin{split}
(\partial_t-\Psidel)u + q(x, u)&=0 \quad \text{in}\; D\times[0, T),
\\
u(x, T)\,=\, g(x)\; \text{and}\; u(x, t)&=\, 0\quad \text{in}\;  D^c\times[0, T].
\end{split}
\end{equation}
\end{lemma}

\begin{proof}
The idea is similar to the elliptic case where we
use monotone iteration method. Let $m$ be Lipschitz constant of 
$s\mapsto q(x, s)$ in $[0, \norm{v}]$, that is,
$$|q(x, s_1)-q(x, s_2)|\leq m |s_1-s_2|\quad \text{for}\; s_1, s_2\in [0, \norm{v}], \; x\in\bar{D}.$$
Let $F(x, s)=q(x, s)+m s$ and $u_0= v_2$.
Define $u_1$ to be the solution of 
\begin{equation*}
\begin{split}
(\partial_t-\Psidel)u_1-mu_1 + F(x, u_0)&=0 \quad \text{in}\; D\times[0, T),
\\
u_1(x, T)\,=\, g(x)\; \text{and}\; u_1(x, t)&=\, 0\quad \text{in}\;  D^c\times[0, T].
\end{split}
\end{equation*}
By Lemma~\ref{L4.2} we then have
\begin{equation}\label{EL4.5B}
u_1(x, t)= \Exp_x\left[e^{-m{(T-t)} }g(X_{T-t})\Ind_{\{T-t<\uptau\}}\right] + 
\Exp_x\left[\int_0^{T-t} e^{-m s} F(X_s, u_0(X_s, t+s))\Ind_{\{s<\uptau\}} \D{s}\right].
\end{equation}
Another use of Lemma~\ref{L4.2} gives
\begin{equation}\label{EL4.5C}
v_i(x, t)= \Exp_x\left[e^{-m(T-t)} v_i(X_{T-t}, T)
\Ind_{\{T-t<\uptau\}}\right] + \Exp_x\left[\int_0^{T-t} e^{-m s} \tilde{F}_i(X_s, v_i(X_s, t+s)) 
\Ind_{\{s<\uptau\}}\D{s}\right],
\end{equation}
where $\tilde{F}_i(x, s)= F(x, s)+\ell_i(x)$. Since $F$ is non-decreasing in $s$, we
have
$$ F(x, v_1)+\ell_1(x)\leq F(x, v_2) \leq F(x, v_2) + \ell_2(x)\,.$$
Therefore, comparing \eqref{EL4.5B} and \eqref{EL4.5C} we have
$v_1\leq u_1\leq u_0=v_2$ in $\Rd\times[0, T]$.
Now we find an iterative sequence of solutions as follows: $u_{k+1}$ is a solution to  
$$(\partial_t -\Psidel) u-m u 
+ F(x, u_k)=0\quad \text{in}\; D\times[0, T), \quad u=0\;\; \text{in}\; D^c\times[0, T],
\quad u(x,T)=g(x).$$
In other words,
\begin{equation}\label{EL4.5D}
u_{k+1}(x, t)= \Exp_x\left[e^{-m{(T-t)} }g(X_{T-t})\Ind_{\{T-t<\uptau\}}\right] + 
\Exp_x\left[\int_0^{T-t} e^{-m s} F(X_s, u_k(X_s, t+s))\Ind_{\{s<\uptau\}} \D{s}\right].
\end{equation}
The above argument shows that 
$$v_1\leq u_{k+1}\leq u_k\leq \cdots\leq v_2\quad \text{in}\; \Rd\times[0, T]\,.$$
Furthermore, applying \cref{L4.4}, we see that $\lim_{k\to\infty} u_k(\cdot, t)= u(\cdot, t)$
uniformly in $x$, for each $t\in[0, T]$. Thus, using dominated convergence theorem, we can
 pass to the limit in \eqref{EL4.5D} to obtain
\begin{equation}\label{EL4.5E}
u(x, t)= \Exp_x\left[e^{-m{(T-t)} }g(X_{T-t})\Ind_{\{T-t<\uptau\}}\right] + 
\Exp_x\left[\int_0^{T-t} e^{-m s} F(X_s, u(X_s, t+s))\Ind_{\{s<\uptau\}} \D{s}\right].
\end{equation}
From \eqref{EL4.5E} it is easy to show that $u$ is continuous in $\Rd\times[0, T]$ (cf. \cite[Lemma~3.1]{BL17a}).
Indeed, since $x\mapsto u(x, t)$ is continuous
uniformly for $t$ in compact subsets of $[0, T)$ and $t\mapsto p_D(t, x, y)$ is continuous in $(0, \infty)$, $(x, t)\mapsto u(x, t)$ is continuous in $[0, T)\times\Rd$. To examine the continuity at $T$ consider
a sequence $(x_n, t_n)\to (x, T)$. Note that the second term in the above display goes to $0$. Again,
if $x\in\partial{D}$ then 
$$\Exp_{x_n}[|g(X_{T-t_n})|\Ind_{\{T-t_n<\uptau\}}]
\leq \Exp_{x_n}[|g(X_{T-t_n})|]\to 0,$$
as $n\to\infty$, we get $u(x_n, t_n)\to 0$.
Also, if $x\in D$, since $p_D(T-t_n, x_n, y)\D{y}\to\delta_x$, we get
$u(x_n, t_n)\to g(x)$. This gives continuity. 
Applying \cref{L4.2} we see that $u$ is a solution to \eqref{EL4.5A}.
Uniqueness of solution follows from \cref{L4.3}. This completes the proof.
\end{proof}

Next we prove a sharp boundary behaviour for the solution of the parabolic equation.
\begin{lemma}\label{L4.6}
Consider $q$ from \cref{L4.3}.
Let $u$ be a bounded solution of 
$$(\partial_t-\Psidel)u + q(x, u)\,=\,0 \quad \text{in}\; D\times[0, T), 
\quad u\,=0\; \text{in}\; D^c\times[0, T],$$
where $u(x, T)\gneq 0$. Then for every $t<T$ there exists a constant $C$, dependent on 
$t, T$ and $u|_{\Rd\times[t, T]}$, satisfying
$$\frac{1}{C}\, V(\delta_D(x))\leq u(x, t)\leq C\, V(\delta_D(x))\quad x\in D\,.$$
\end{lemma}

\begin{proof}
Denote by 
$$H(x, t)=\frac{q(u(x, t))}{u(x, t)}.$$
Then $H$ is a bounded, continuous function.
Using \cref{L4.2} we then have
$$u(x, t) =
\Exp_x\left[e^{\int_0^{T-t} H(X_s, t + s)\, \D{s}}
u(X_{T-t}, T)\Ind_{\{T-t<\uptau\}}\right].$$
Thus, for some constant $C_1$, get
\begin{equation}\label{EL4.6A}
e^{-C_1 T} \Exp_x\left[u(X_{T-t}, T)\Ind_{\{T-t<\uptau\}}\right]
\leq u(x, t)\leq e^{C_1 T} \Exp_x\left[u(X_{T-t}, T)\Ind_{\{T-t<\uptau\}}\right].
\end{equation}
Using \eqref{EL4.4C} and \eqref{EL4.6A} we obtain
$$u(x, t)\leq C\, V(\delta_D(x)),$$
which gives the upper bound.
Now from \cite[Theorem~4.5]{BGR14} we know  that
$$p_D(t, x, y)\geq \kappa \Prob_x(\uptau>t/2)\,\Prob_y(\uptau>t/2)\, p(t\wedge V^2(r), |x-y|)$$
and 
$$ \Prob_x(\uptau>t/2)\geq \kappa \left(\frac{V(\delta_D(x))}{\sqrt{t\wedge V(r)}}\wedge 1\right),$$
where $D$ satisfies the inner and outer ball condition with radius $r$.
Now let $\cK\Subset D$ be such that $\min_{\cK}u(x, T)\geq \kappa_2>0$.
Using the lower bound in \eqref{EL4.6A} and estimates above we then find
\begin{align*}
u(x, t)&\geq e^{-C_1 T} \int_D u(y, T)\,p_D(T-t, x, y) \D{y}
\\
&\geq C_2 \kappa \Prob_x(\uptau>(T-t)/2) \int_\cK u(y, T) p((T-t)\wedge V^2(r), |x-y|) \Prob_y(\uptau>(T-t)/2) 
\D{y}
\\
&\geq C_3 \kappa^2 \kappa_2\, V(\delta_D(x))\;  p((T-t)\wedge V^2(r), \diam(D))\int_\cK \Prob_y(\uptau>(T-t)/2) dy
\\
&\geq C^{-1}\, V(\delta_D(x)),
\end{align*}
for some constants $C_2, C_3, C$.
This gives the lower bound.
Hence the proof.
\end{proof}

Now we are ready to prove \cref{T1.3}. Recall that given an interval $[0, T]$,
$u_T$ solves
\begin{equation}\label{E4.21}
\begin{split}
(\partial_t-\Psidel) u_T + a u_T - f(x, u_T) &=\, 0 \quad \text{in}\; D\times [0, T),
\\
u_T(x, T)\,=\, u_0(x)\;  \text{and}\; u_T(x, t)&=\, 0\quad \text{in}\;  D^c\times[0, T],
\end{split}
\end{equation}
where $0\lneq u_0\in\cC_0(D)$.

\begin{proof}[{\bf Proof of \cref{T1.3}}]
First consider (i). We divide the proof in two steps.

{\bf Step 1.} First we note that 
$$u_T(x, T-1)= \Exp_x\left[\Ind_{\{1<\uptau\}}
u_0(X_{1})\right] + \Exp_x\left[\int_0^{1} \Ind_{\{s<\uptau\}}  F(X_s, u_T(X_s, T-1+s)) \,\D{s}\right],$$
where $F(x, s)=as-f(x, s)$.
Thus $u_T(x, T-1)$ is independent of $T$ (by Lemma~\ref{L4.3}). In fact, it is same  as $v(x, 0)$ where $v$ solves \eqref{E4.21} in $[0, 1]$. Also,
from \eqref{EL4.2B} (taking $\delta=T-1-t$) we note that for $t\leq T-1$ we have
\begin{align*}
u_T(x, t)
&=\Exp_x\left[\Ind_{\{T-1-t<\uptau\}}
u_{T}(X_{T-1-t}, T-1)\right] + \Exp_x\left[\int_0^{T-1-t} \Ind_{\{s<\uptau\}}  F(X_s, u_T(X_s, t+s)) \D{s}\right]\,.
\end{align*}
Thus without any less of generality we may assume $u_0=u_{T}(x, T-1)$. In particular,
by \cref{L4.6}, we obtain
\begin{equation}\label{ET1.3A}
C^{-1}\, V(\delta_D(x))\leq u_0(x)\leq C\, V(\delta_D(x)), \quad \text{for}\; x\in D.
\end{equation} 

{\bf Step 2.} Let $v_a$ be the unique positive solution (see \cref{T1.1}) to 
\begin{equation}\label{ET1.3B}
-\Psidel v_a + av_a-f(x, v_a)=0\quad \text{in}\; D, \quad v_a=0\quad \text{in}\; D^c, \quad v_a>0 \quad \text{in}\; D\,.
\end{equation}
Using \eqref{ET1.3A}, \cref{T2.1,T2.2}, we choose $\kappa>1$ large enough so that
$$
\breve\varphi(x)\df\kappa^{-1} v_a(x)\leq u_0(x)\leq \kappa v_a(x)\df\hat\varphi(x),
\quad x\in D\,.
$$
Note that $\breve\varphi$ is subsolution to \eqref{ET1.3B} and $\hat\varphi$ is a supersolution to \eqref{ET1.3B}.
Setting $\hat\varphi$ as the terminal condition at time $T$ we construct a solution $\hat{w}_T$ in $[0, T]$
with $\hat{w}_T\leq \hat\varphi$. This can be done using \cref{L4.5}. Next we observe that $\hat{w}$ is 
increases with $t$. For instance, take $t_1\leq t_2\leq T$ with $T-t_2=t_2-t_1$.
Observe that $\xi(x, t)=\hat{w}_T(x, t-t_2+t_1)$ is a solution to 
\begin{equation*}
\begin{split}
(\partial_t-\Psidel) u + a u - f(x, u) &=\, 0 \quad \text{in}\; D\times [t_2, T),
\\
u(x, T)\,=\, u_T(x, t_2)\;  \text{and}\; u(x, t)&=\, 0\quad \text{in}\;  
D^c\times[t_2, T].
\end{split}
\end{equation*}
Using the uniqueness of
solutions  and comparison principle (\cref{L4.3}) we 
see that $\hat{w}_T(x, t_1)\leq \hat{w}_T(x, t_2)$. For any pair $t_1\leq t_2\leq T$
the same comparison holds due to continuity with respect to $t$ and a density
argument.
Another application of  \cref{L4.3} gives
that $u_T(x, 0)\leq \hat{w}_T(x, 0)\leq \hat\varphi(x)$. Now apply Lemma~\ref{L4.4} to invoke equi-continuity
and show that $\hat{w}_T(x, 0)\to \hat{w}$ as $T\to\infty$. Then passing limit in 
$$\hat{w}_T(x, 0)= \Exp_x\left[\Ind_{\{T<\uptau\}}
\hat\varphi(X_{T})\right] + \Exp_x\left[\int_0^{T} \Ind_{\{s<\uptau\}}  F(X_s, \hat{w}_T(X_s, s)) ds\right],$$
as $T\to\infty$, we obtain
$$\hat{w}(x)= \Exp_x\left[\int_0^{\uptau}  F(X_s, \hat{w}(X_s, s)) ds\right]
=\cG F(\cdot, \hat{w}) (x).$$
This, in particular, implies 
$$-\Psidel\hat{w} + a\hat{w}-f(\hat{w})=0\quad \text{in}\; \Omega, \quad \hat{w}=0\quad \text{in}\; \Omega^c.$$
From uniqueness we must have $\hat{w}=v_a$.

Follow a similar argument to construct a sequence of solution $\breve{w}$ (decreasing in $t$) satisfying
$$\breve\varphi\leq \breve{w}_T(x, 0)\leq u_{T}(x, 0).$$
Argument similar to above shows that
$$\lim_{T\to\infty}\sup_{D}|\breve{w}_T(x, 0)-v_a|=0.$$
Combining these two observations we complete the proof of (i).

(ii) Proof is similar to (i). For $a\leq \lambda_1$, we take $\varphi_1$ as the super-solution to \eqref{ET1.3B}. Then repeating a same argument we can conclude the proof.
\end{proof}

\subsection*{Acknowledgments}
This research of Anup Biswas was supported in part by DST-SERB grants EMR/2016/004810 and MTR/2018/000028. Mitesh Modasiya is partially supported by CSIR PhD fellowship (File no.\\
09/936(0200)/2018-EMR-I).

\end{document}